\documentclass[10pt]{amsart}

\usepackage{amsmath,amsthm,amssymb,latexsym,amstext,amsfonts,mathabx}
\usepackage[english]{babel}
\usepackage{geometry}
\usepackage{graphics,comment}
\usepackage{graphicx}
\usepackage{helvet}
\usepackage{enumitem}
\usepackage{pifont}
\usepackage{ccfonts} 
\usepackage{bbold}
\usepackage{tikz}
\usepackage{tikz-cd}
\usetikzlibrary{%
  matrix,%
  calc,%
  arrows,%
  shapes,%
  positioning%
}
\usetikzlibrary{positioning}
\usepackage{hyperref}
\usepackage{caption}
\usepackage{subfigure}

\usepackage{lineno}

\newtheorem{theorem}{Theorem}[subsection]

\newtheorem{corollary}[theorem]{Corollary}

\newtheorem{definition}[theorem]{Definition}

\newtheorem{example}[theorem]{Example}

\newtheorem{lemma}[theorem]{Lemma}

\newtheorem{problem}[theorem]{Problem}
\newtheorem{proposition}[theorem]{Proposition}
\newtheorem{remark}[theorem]{Remark}

\newtheorem{contributions}[theorem]{Contributions}

\newcommand{\R}{\mathbb{R}}

\newcommand{\G}{\mathbb{G}}
\newcommand{\Ff}{\mathbb{F}}

\newcommand{\X}{\mathbb{X}}

\def \GG{{\mathcal G}}
\def \HH{{\mathcal H}}

\def \DD{{\mathcal D}}

\def \FF{{\mathcal F}}
\def \CC{{\mathcal C}}

\def \FF{{\mathcal F}}

\def \Pe{{\mathcal P}}

\def \VV{{\mathcal V}}

\def \EE{{\mathcal E}}

\DeclareMathOperator{\im}{im}

\DeclareMathOperator*{\colim}{colim}

\newcommand{\set}[1]{\{\,#1\,\}}
\newcommand{\x}{\mathbb{x}}
\newcommand{\Hg}{\mathrm{H}}
\newcommand{\rto}{\rightarrow}
\newcommand{\from}{\leftarrow}
\newcommand{\kk}{\mathrm{k}}
\newcommand{\meet}{\wedge}
\newcommand{\join}{\vee}
\newcommand{\bigmeet}{\bigwedge}
\newcommand{\bigjoin}{\bigvee}

\newcommand\dfn{\textbf}

\begin{document}

\title{\textbf{A LATTICE FOR PERSISTENCE}}

\author{Jo\~ao Pita Costa and Primo\v z \v Skraba}
\address{In\v stitut Jo\v zef Stefan,\\
Jamova Cesta 39, 1000 Ljubljana, Slovenia.
}
\date{\today}

\maketitle

\linenumbers


\begin{abstract} 
The intrinsic connection between lattice theory and topology is fairly well established.
For instance, the collection of open subsets of a topological subspace always forms a distributive lattice. 
Persistent homology has been one of the most prominent areas of research in computational topology in the past 20 years.
In this paper we will introduce an alternative interpretation of persistence based on the study of the order structure of its correspondent lattice. 
Its algorithmic construction leads to two operations on homology groups which describe an input diagram of spaces as a complete Heyting
algebra, which is a generalization of a Boolean algebra.  
We investigate some of the properties of this lattice, the algorithmic implications of it, and some possible applications.
\end{abstract}


\bigskip

\renewcommand{\contentsname}{Table of Contents}

\tableofcontents 

\bigskip


\section*{Introduction}
\label{Introduction}

Persistent (co)homology is one of the central objects of study in applied and computational topology~\cite{Edel00}. Numerous extensions have been proposed to the original formulation including zig-zag persistence~\cite{ZigZag} and multidimensional persistence~\cite{Carl09}, whereas the original persistence looks at a filtration (i.e., an increasing sequence of spaces). Zig-zag persistence extended the theory and showed that the direction of the maps does not matter, using tools from quiver theory. In multidimensional persistence, multifiltrations are considered. In this paper, we also look at the problem of persistence in more general diagrams of spaces using tools from lattice theory. There is another key difference in this work however. Rather than try to find a decomposition of the diagram of spaces into indecomposables, we concentrate on pairs of spaces within diagrams addressing the more difficult problem of indecomposables in the sequel paper. 

Lattice theory is the study of order structures. The deep connections between topology and lattice theory has been known since the work of Stone~\cite{Joh86}, showing a duality between Boolean algebras and certain compact and Hausdorff topological spaces, called appropriately \emph{Stone spaces}.  
In the first section of this paper we present the basic concepts of lattice theory.
These preliminaries mostly refer to classical results on distributive lattices and Heyting algebras, and can be skipped by the reader that is familiar with the subject.
A study of lattice theory and, in general, of universal algebra, can be found in \cite{Ba40}, \cite{Sa81}, \cite{Gr71} and \cite{Gr79}.

A description of the topological background follows in the second section, reviewing the main concepts and results of Persistent Homology and suggesting several examples that are a motivation to this study. 
Good reviews on topological data analysis are given in \cite{TD} and \cite{Zom05}, on persistent homology are given in \cite{Skr13} and \cite{Zom13}, and on zig-zag persistence are given in \cite{ZigZag}, \cite{Car09} and \cite{Oud12}.

In the following section we describe the order structure of our input diagram of spaces by a partial order induced by certain maps between vector spaces, and show that this order provides a lattice structure.
We construct the meet and join operations using the natural concepts of limits and colimits of linear maps, and show that this construction stabilizes.  
We shall see that the constructed lattice is a complete Heyting algebra, one of the algebraic objects of biggest interest in topos theory.

From the latter results we discuss connections with persistent homology, and give a different perspective on several aspects of this theory. 
In particular, we look at diagrams of spaces and retrieve general laws both based on concrete examples (like standard or zig-zag persistence) and on the interpretation of laws derived from the lattice theoretic analysis. Finally  we introduce a few algorithmic applications which we will develop further in a subsequent paper.


%
\section{Preliminaries}
\label{Preliminaries}

A \emph{lattice} is a partially-ordered set (or poset) expressed by $(L,\leq)$  for which all pairs of elements have an infimum and a supremum, denoted by $\wedge$ and $\vee$, respectively, commonly known as the \emph{meet} and \emph{join} operations. 
The lattice properties correspond to the minimal structure that a poset must have to be seen as an algebraic structure. 
Such algebraic structure $(L;\wedge ,\vee)$ is given by two operations $\wedge$ and $\vee$ satisfying:
\begin{enumerate}
\item[L1.] \emph{associativity}: $x\wedge (y\wedge z) = (x \wedge y)\wedge z$ and $x\vee (y\vee z) = (x \vee y)\vee z$, 
\item[L2.] \emph{idempotency}: $x\wedge x = x = x\vee x$, 
\item[L3.] \emph{commutativity}: $x\wedge y=y\wedge x$ and $x\vee y=y\vee x$ 
\item[L4.]  \emph{absorption}: $x\wedge (x\vee y)=x=x\vee (x\wedge y)$.
\end{enumerate}
The equivalence between this algebraic perspective of a lattice $L$ and its ordered perspective is given by the following equivalence: for all $x,y\in L$, $x\leq y$ iff $x\wedge y=x$ iff $x\vee y=y$. 
At that stage the order and the algebraic structures hold the same information over different perspectives. 
If every subset of a lattice $L$ has a supremum and an infimum, $L$ is named a \emph{complete lattice}.
All finite lattices are complete. 
A partial order is named \emph{total order} if every pair of elements is related, that is, for all $x,y\in A$, $x\leq y$ or $y\leq x$. 
On the other hand, an \emph{antitotal order} is a partial order for which no two elements are related. 
Examples of lattices include the power set of a set ordered by subset inclusion, or the collection of all partitions of a set ordered by refinement.
Every lattice can be determined by a unique undirected graph for which the vertices are the lattice elements and the edges correspond to the partial order: the \emph{Hasse diagram} of the lattice. 
With additional constraints on the operations we get different types of lattices.  
In particular, a lattice $L$ is \emph{distributive} if, for all $x,y,z\in S$, it satisfies one of the following equivalent equalities:

\begin{itemize}
\item[(d1)] $x\wedge (y\vee z)=(x\wedge y)\vee (x\wedge z)$; 
\item[(d2)] $x\vee (y\wedge z)=(x\vee y)\wedge (x\vee z)$;
\item[(d3)] $(x\vee y) \wedge (x\vee z) \wedge (y\vee z) = (x\wedge y)\vee (x\wedge z)\vee (y\wedge z)$. 
\end{itemize}

The lattice of subsets of a set ordered by inclusion is a distributive lattice.
The lattice of normal subgroups of a group as well as the lattice of subspaces of a vector space are not distributive (cf. \cite{Ba40}).
A lattice $L$ is distributive if and only if for all $x,y,z\in L$, $x\wedge y=x\wedge z$ and $x\vee y=x\vee z$ imply $y=z$ (\cite{Ba40}).
A \emph{Boolean algebra} is a distributive lattice with a unary operation $\neg$ 
and nullary operations $0$ and $1$ such that for all elements $a\in A$, $a\vee 0 = a$ and $a\wedge 1=a$ as $a\vee \neg a = 1$ and  $a\wedge \neg a = 0$.
While the power of a set with intersection and union is a Boolean algebra, total orders are examples of distributive lattices that are not Boolean algebras in general.
A bounded lattice $L$ is a \emph{Heyting algebra} if, for all $a,b\in  L$ there is a greatest element $x\in L$ such that $a\wedge x\leq b$. This element is the \emph{relative pseudo-complement} of $a$ with respect to $b$ denoted by $a\Rightarrow b$. 
Examples of Heyting algebras are the open sets of a topological space, as well as all the finite nonempty total orders (that are bounded and complete). 
Furthermore, every complete distributive lattice $L$ is a Heyting algebra with the implication operation given by $x \Rightarrow y = \bigvee \set{x\in L \mid x \wedge a \leq b}$.

\begin{contributions}
Universal algebra and lattice theory, in particular, are transversal disciplines of Mathematics and have proven to be of interest to the study of any algebraic structure. In the following sections we will describe the construction of a lattice completing a given commutative diagram of homology groups. We will show that this lattice is complete and distributive, thus constituting a complete Heyting algebra. Despite the nice algebraic properties that hold in this structure as a consequence of being such an algebra, it does not constitute a Boolean algebra.
\end{contributions}


%
\section{Problem Statement}
\label{Problem Statement}

We assume a basic familiarity with algebraic topological notions such as (co)homology, simplicial complexes, filtrations, etc. 
For an overview, we recommend the references \cite{Hat00} for algebraic topology, as well as \cite{Ed10} and \cite{Zom05} for applied/computational
topology.
We motivate our constructions with the examples in the following paragraphs.

Consider persistent homology, presented in \cite{Edel00}. 
Let $\X$ be a space and $f:\X \rto \R$ a real function.  
The object of study of persistent homology is a filtration of $\X$, i.e., a monotonically non-decreasing sequence
\begin{equation*}
\emptyset = \X_0 \subseteq \X_1 \subseteq \X_2 \subseteq \ldots\subseteq\X_{N-1} \subseteq \X_N = \X
\end{equation*}
To simplify the exposition, we assume that this is a discrete finite filtration of tame spaces. Taking the homology of each of the associated chain complexes, we obtain 
\begin{equation*}
 \Hg_*(\X_0) \rto  \Hg_*(\X_1) \rto  \Hg_*(\X_2) \rto \ldots\rto \Hg_*(\X_{N-1}) \rto  \Hg_*(\X_N)
\end{equation*}
We take homology over a field $\kk$ -- therefore the resulting homology groups are vector spaces and the induced maps are linear maps. In~\cite{Edel00}, the $(i,j)-$\emph{persistent homology groups}  of the filtration are defined as
\begin{equation*} 
\Hg^{i,j}_*(\X) = \im (\Hg_*(\X_i) \rto   \Hg_*(\X_j))
\end{equation*}

This motivates the idea for the construction of a totally ordered lattice.
To see this, let us consider the set of the homology groups with a partial order induced by the indexes of the spaces in the filtration. We can define two lattice operations $\meet$ and $\join$ as follows:
\begin{itemize}
\item[] $\Hg_*(\X_i) \join \Hg_*(\X_j) = \Hg_*(X_{\max(i,j) })$
\item[] $\Hg_*(\X_i) \meet \Hg_*(\X_j) = \Hg_*(X_{\min(i,j) })$
\end{itemize}
With these operations we get a finite total order and, thus, a complete Heyting algebra (see this discussion in the following section). 
The definition of persistent homology groups can then be rewritten as follows:

\begin{definition}\label{rank}
For any two elements $\Hg_*(\X_i)$ and $ \Hg_*(\X_j)$, the rank of the persistent homology classes is 
\[
\im (\Hg_*(\X_i\meet\X_j) \rto \Hg_*(\X_i\join \X_j)).
\]
\end{definition}


The case of a filtration, where a total order exists, does not have a very interesting underlying order structure. 
Let us now look at the case where we have more than one parameter. 
We define a \emph{diagram} to be a directed acyclic graph of vector spaces (vertices) and linear maps between them (edges). 
This is known as multidimensional persistence and has been studied in \cite{Carl09} and \cite{Zom09}. 
We shall start by looking at a bifiltration, i.e., a filtration on two dimensions (or parameters).
Observe that, for related elements of the filtration, these operations coincide with the ones defined above for the standard persistence case. 
However, when we consider incomparable elements, the meet and join operations are given by the rectangles they determine. 
Adjusting our definitions from above we can define the lattice operations in a natural way by setting:

\begin{itemize}
\item[] $\Hg_*(\X_{i,j}) \join \Hg_*(\X_{k,\ell}) = \Hg_*(X_{\max(i,k),\max(j,\ell) })$
\item[] $\Hg_*(\X_{i,j}) \meet \Hg_*(\X_{k,\ell}) = \Hg_*(X_{\min(i,k),\min(j,\ell) })$
\end{itemize}

Consider the bifiltration of dimensions $4\times 4$ from Figure \ref{figmulti}. The Hasse diagram of the correspondent underlying algebra is presented in Figure \ref{figmultidim}.
In that diagram, $\X_{01}\leq \X_{31}$ and clearly, $\X_{01}\meet \X_{31}=\X_{01}$ while $\X_{01}\join \X_{31}=\X_{31}$.
On the other hand, $\X_{02}$ and $\X_{11}$ are unrelated with $\X_{02}\meet \X_{11}=\X_{01}$ while $\X_{02}\join \X_{11}=\X_{12}$.
Note that, by the commutativity of the diagram, any two elements which have the same meet and join define the same rectangle in the bifiltration, determined by the properties in the Hasse diagrams represented in Figure \ref{figmultidim}. 
By the assumed commutativity of the diagram of spaces, any path through the rectangle has equal rank and so the map of the meet to join gives the rank invariant of Definition \ref{rank}.

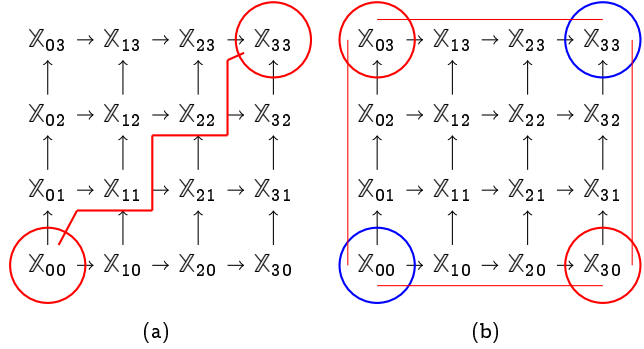
\begin{figure}

\subfigure[]{

\begin{tikzpicture} 
         \foreach \i in {0,...,3}{
           \foreach \j  in {0,...,3}{
             \node (p\i\j) at (\i,\j) {$\X_{\i\j}$}    ;
           }
         }
         \foreach \i [evaluate=\i as \x using int(\i+1)] in {0,...,2}{
           \foreach \j in {0,...,3}{
             \draw[->] (p\i\j) -- (p\x\j)  ;
           }
         }
     \foreach \i  in {0,...,3}{
           \foreach \j [evaluate=\j as \y using int(\j+1)] in {0,...,2}{
             \draw[->] (p\i\j) -- (p\i\y)  ;
           }
         }
\draw[thick, red] (p00) circle (0.5cm);
\draw[thick, red] (p33) circle (0.5cm);
\draw[thick,red] (p00) -- (p01.south east);

\draw[thick,red] (p01.south east) -- (p11.south east); 
\draw[thick,red] (p11.south east) -- (p12.south east); 
\draw[thick,red] (p12.south east) -- (p22.south east);
\draw[thick,red] (p22.south east) -- (p23.south east);
\draw[thick,red] (p23.south east) -- (p33);

\end{tikzpicture}
}
\subfigure[]{

\begin{tikzpicture} 
         \foreach \i in {0,...,3}{
           \foreach \j  in {0,...,3}{
             \node (p\i\j) at (\i,\j) {$\X_{\i\j}$}    ;
           }
         }
         \foreach \i [evaluate=\i as \x using int(\i+1)] in {0,...,2}{
           \foreach \j in {0,...,3}{
             \draw[->] (p\i\j) -- (p\x\j)  ;
           }
         }
     \foreach \i  in {0,...,3}{
           \foreach \j [evaluate=\j as \y using int(\j+1)] in {0,...,2}{
             \draw[->] (p\i\j) -- (p\i\y)  ;
           }
         }
\draw[thick, red] (p30) circle (0.5cm);
\draw[thick, red] (p03) circle (0.5cm);

\draw[thick, blue] (p00) circle (0.5cm);
\draw[thick, blue] (p33) circle (0.5cm);

\draw[red] (p00.west) -- (p03.west);
\draw[red] (p00.south) -- (p30.south);
\draw[red] (p03.north) -- (p33.north);
\draw[red] (p30.east) -- (p33.east);

\end{tikzpicture}
}
\caption{The lattice operations in the case of a bifiltration. (a) If
  the two elements are comparable, by the commutativity of the diagram
  we can choose any path to find the persistent homology groups. (b)
  If the elements are incomparable, we can find the smallest and
  largest elements where they become comparable. In both cases we
  recover the rank invariant of \cite{Carl09} }
\label{figmulti}
\end{figure}

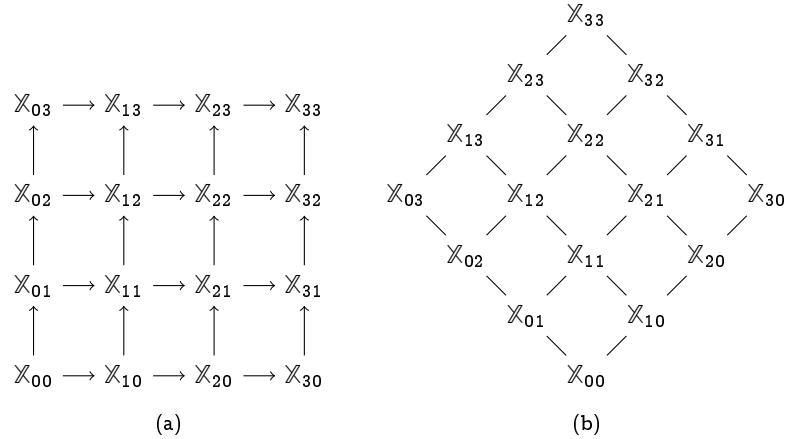
\begin{figure}
\centering
\subfigure[]{%
     \begin{tikzpicture}[scale=.6] 
         \foreach \i  [evaluate=\i as \x using int(2*\i)] in {0,...,3}{
           \foreach \j [evaluate=\j as \y using int(2*\j)] in {0,...,3}{
             \node (p\i\j) at (\x,\y) {$\X_{\i\j}$}    ;
           }
         }
         \foreach \i [evaluate=\i as \x using int(\i+1)] in {0,...,2}{
           \foreach \j [evaluate=\j as \y using int(\j+1)]  in {0,...,2}{
             \draw[->] (p\i\j) -- (p\i\y)  ;
             \draw[->] (p\i\j) -- (p\x\j)  ;
           }
         }
         \foreach \j [evaluate=\j as \y using int(\j+1)]  in {0,...,2}{
           \draw[->] (p3\j) -- (p3\y)  ;
           \draw[->] (p\j3) -- (p\y3)  ;
           }
     \end{tikzpicture}
}     
 \quad
 \subfigure[]{%
     \begin{tikzpicture}[scale=.8]

  \node (03) at (-3,1){$\X_{03}$} ;
  \node (13) at (-2,2){$\X_{13}$} ;
  \node (23) at (-1,3){$\X_{23}$} ;
  \node (33) at (0,4){$\X_{33}$} ;
  \node (32) at (1,3){$\X_{32}$} ;  
  \node (31) at (2,2){$\X_{31}$} ;
  \node (30) at (3,1){$\X_{30}$} ;  
  \node (22) at (0,2){$\X_{22}$} ;
  \node (12) at (-1,1) {$\X_{12}$} ; 
  \node (21) at (1,1){$\X_{21}$} ;
  \node (02) at (-2,0){$\X_{02}$} ;
  \node (20) at (2,0){$\X_{20}$} ;
  \node (11) at (0,0) {$\X_{11}$} ;
  \node (01) at (-1,-1){$\X_{01}$} ;
  \node (10) at (1,-1) {$\X_{10}$} ;
  \node (00) at (0,-2) {$\X_{00}$} ;
  \draw (00) -- (01) -- (02) -- (03) -- (13) -- (23) -- (33) -- (32) -- (31)-- (30) -- (20) -- (10) -- (00);
  \draw (02) -- (12) -- (22) -- (32);
  \draw (01) -- (11) -- (21) -- (31);
  \draw (10) -- (11) -- (12) -- (13);      
  \draw (20) -- (21) -- (22) -- (23);  
\end{tikzpicture}
}
\caption{The diagram of a bifiltration of dimensions $4\times 4$ (a) and the Hasse diagram of the correspondent underlying Heyting algebra (b).}
\label{figmultidim}       
\end{figure}     

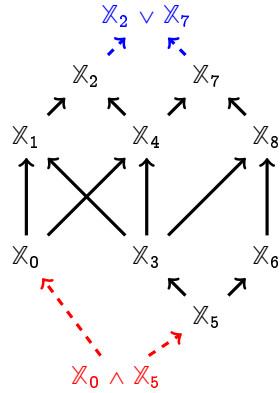
\begin{figure}

\begin{center}
\begin{tikzpicture}[very thick, scale=0.8]
\node (a) at (0,0) {$\X_0$};
\node (b) at (0,2) {$\X_1$};
\node (c) at (1,3) {$\X_2$};
\node (d) at (2,0) {$\X_3$};
\node (e) at (2,2) {$\X_4$};
\node (f) at (3,-1) {$\X_5$};
\node (g) at (4,0) {$\X_6$};
\node (i) at (3,3) {$\X_7$};
\node (h) at (4,2) {$\X_8$};
\draw[->] (a) -- (b);
\draw[->] (b) -- (c);
\draw[->] (a) -- (e);
\draw[->] (d) -- (e);
\draw[->] (d) -- (b);
\draw[->] (e) -- (c);

\draw[->] (e) -- (i);
\draw[->] (h) -- (i);
\draw[->] (g) -- (h);
\draw[->] (f) -- (g);
\draw[->] (f) -- (d);

\draw[->] (d) -- (h);

\node (e1) at (1.5,-2) {\textcolor{red}{$\X_0 \wedge \X_5$}};
\draw[->,red, dashed] (e1) -- (f) ;
\draw[->,red, dashed] (e1) -- (a) ;

\node (e2) at (2,4) {\textcolor{blue}{$\X_2 \vee \X_7$}};
\draw[->,blue, dashed] (i) -- (e2) ;
\draw[->,blue, dashed] (c) -- (e2) ;

\end{tikzpicture}
\end{center}

\caption{General commutative diagrams of spaces and linear maps between them.}
\label{gendigm}
\end{figure}

Both of these cases are highly-structured. 
Consider the case of a more general diagram of homology groups in Figure \ref{gendigm}.
While we can embed this diagram in a multifiltration, by
augmenting the diagram with $0$ and unions of space, however the result is not very informative.
The defined lattice operations can bring a complementary knowledge to this study.
This is the motivation for the construction we present in this
paper. Since we deal with homology over a field, we look to analyze more
general but commutative diagrams of vector spaces. 

\begin{problem}
Given a commutative diagram of vector spaces and linear maps between
them, we  construct an order structure that completes it into a lattice, study its 
algebraic properties and develop algorithms based on this.  
\end{problem}

\begin{remark}
Quiver theory is also concerned with diagrams of vector spaces and
linear maps. However, a key difference is that the diagrams in quiver
theory are generally not required to be commutative. 
\end{remark}

\begin{remark}
We concentrate on the persistence between two elements rather than decomposition of the entire diagram. 
While we believe the constructions in this paper can aid this
decomposition, it does not immediately follow. As such, any reference
to a diagram should be understood as referring to the input collection
of vector spaces and linear maps, corresponding to the partial Hasse diagram of the underlying lattice structure, 
rather than a persistence diagram.
\end{remark}


%
\section{Lattice Structure}
\label{Structural Theory}

Here we introduce how to retrieve the order information from a diagram of vector spaces and linear maps, and construct the lattice operations determined by that order, where the elements are vector spaces. 
The linear maps between them will define the relations between those vector spaces and limit concepts like equalizers and coequalizers (roughly, an equalizer is a solution set of equations while a coequalizer is a generalization of a quotient by an equivalence relation) will serve us to define biggest and least elements.   

\subsection{The Lattice Operations}


Consider a diagram of vector spaces and linear maps and assume one unique component.
The underlying ordered structure is a poset defined as follows: 

\begin{definition}
For all vector spaces $A$ and $B$ of a given diagram $\DD$,
\[
A\leq B\text{  if there exists a linear map  }f:A\rightarrow B.
\]
The partial order $\leq $ is, thus, the set of ordered pairs correspondent to the linear maps in the commutative diagram of spaces given as input.
The identity map ensures the reflexivity of the relation: for all vector spaces $A$ the identity map $id_A$ provides the endorelation $\lefttorightarrow A$.
Transitivity is given by the fact that the composition of linear maps is a linear map and by the assumption that all diagrams are commutative. 
Antisymmetry is given by the fact that $A\leftrightarrows B$ implies $A\leftrightsquigarrow B$, that is, $A$ and $B$ are equal up to isomorphism: 
in detail, having the identity morphisms and usual composition of linear maps, the existence of linear maps $f:A\rightarrow B$ and $g:B\rightarrow A$ imply that $g\circ f=id_{A}$ and that $f\circ g=id_{B}$, as required.
This partial order does not yet have to constitute a lattice but will be completed into one, due to the following constructions. 
The extension of the partial order $\leq$ will be noted by the same symbol, being a part of that bigger partial order.
\end{definition}

\begin{remark}
We consider the object under study to be a commutative diagram of vector spaces and linear maps. As vector spaces are determined up to isomorphism by rank, the equivalence deserves some additional comments. As described above, the reverse maps exist in the case of isomorphisms. This further ensures that the poset structure is well-defined since we cannot arbitrarily reverse the direction of the arrows (as is often the case in representation theory, where the direction of arrows often does not matter). 
If we were to reverse an arrow with a non-unique (but equal rank) map, it is clear that the composition will not commute with identity unless the map is an isomorphism. 
Likewise,  for equivalence we not only require the vector spaces to be isomorphic (of the same rank) but also that there exists a composition of maps in the diagram (possibly including inverses) for which an isomorphism exists. 
Note that this does not imply that all the maps must be isomorphisms. 
\end{remark}

%
%
%
%
In the following paragraphs we will describe the construction of the operations $\wedge $ and $\vee$ over a given diagram $\DD$ of vector spaces and linear maps.
The construction of these lattice operations is based on the concept of direct sum, and the categorical concepts of \emph{limit} and \emph{colimit}. 
In particular, it is based in the generalized notions of \emph{equalizer} and \emph{coequalizer} that we describe right away.
See the details of some of these constructions in Appendix \ref{Algebraic Constructions}. 
%
%
As we assume that all diagrams of vector spaces commute, the categorical concepts of \emph{equalizer} and \emph{coequalizer} can be adapted to the framework of this paper in the following way:

\begin{definition}\label{coeq}
Given a pair of vector spaces $A$ and $C$ with two linear maps $f,g:A\Rightarrow B$ between them: 
\begin{itemize}
\item[(i)] the \emph{equalizer} of $f$ and $g$ is a pair $(E, e)$ where $E$ is a vector space (usually called \emph{kernel set} of the equalizer) and $e:E\rightarrow A$ is a linear map such that $fe=ge$, for any other vector space $E'$ and linear map $e' : E' \rightarrow A$ there exists a unique linear map $\phi : E' \rightarrow E$.
\item[(ii)] the $\emph{coequalizer}$ of $f$ and $g$ is a pair $(H, h)$ where $H$ is a vector space (usually called the \emph{quotient set} of the coequalizer) and $h:A\rightarrow H$ is a linear map such that, for any other vector space $H'$ and linear map $h' : A \rightarrow H'$ there exists a unique linear map $\phi : H \rightarrow H'$.
\end{itemize}
\end{definition}

\begin{remark}
The intuitive idea of looking at the equalizer of two maps $f$ and $g$ as the solution set of the equation $f(x)=g(x)$ in the appropriate domain, is extended to a solution set of several equations. 
Indeed, any system of equations can be seen as one unique (matrix) equation with all the equations that it is constituted being considered as vectors in this matrix. 
Dual remarks hold for coequalizers of more than two maps.

The (co)equalizer is sometimes identified with the kernel set (quotient set). 
Both the concepts of equalizer and coequalizer can be generalized to comprehend the equalization of more than two maps which corresponds to a solution set of several equations. 
Given vector spaces $A$, $B$, $C$ and $D$, with linear maps $f_A:A\rto C$, $f_B:B\rto C$, $g_A:D\rto A$ and $g_B:D\rto B$ we can express these relations by the linear maps $f:A\oplus B\rto C$ and $g:D\rto A\oplus B$ without loss of information. 
If $\FF=\set{f,g,h,\dots}$ its equalizer may be written as $eq(f,g,h,\dots )$ while its coequalizer is written as $coeq(f,g,h,\dots )$.
For the sake of intuition, the kernel set can be thought of as the space of solutions of all the equations determined by the linear maps that are equalized, while the quotient set of a coequalizer can be thought of as the space of constraints that an equation must satisfy, as the space of obstructions, regarding the equations determined by the considered linear maps.
Indeed, for modules over a commutative ring, the equalizer of $f$ and $g$ is $ker(f-g)$ while their coequalizer is $coker(f-g)=B/\im(f-g)$. 
This and other topics are discussed in detail in the appendix of this paper.  
\end{remark}

\begin{definition}
A vector space is a \emph{source} if it is no codomain of any map, and dually it is a \emph{target} if it is no domain of any map (corresponding to the categorical concepts of \emph{initial element} and \emph{terminal element}, respectively. 
Moreover, we call \emph{common source} of a collection of spaces $D_i$ in the given diagram $\DD$, a space $D\in \DD$ mapping in $\DD$ to each of the spaces $D_i$.
Dually, we call \emph{common target} of the collection $D_i$ to a space $D\in \DD$ such that each $D_i$ maps to $D$.
\end{definition}

\begin{remark}\label{exist}
Given vector spaces $X$, $Y$, $Z$ and $W$ in a diagram $\DD$,

\begin{itemize}
\item[(i)] if $Z$ is a common target of $X$ and $Y$ then $Z$ is a target of $X\oplus Y$;
\item[(ii)] if $W$ is a common source of $X$ and $Y$ then $W$ is a source of $X\oplus Y$.
\end{itemize}
While $(i)$ follows from the fact that the direct sum is the coproduct in the category of vector spaces and linear maps, to see $(ii)$ consider the inclusion maps $i_{X}:X\rightarrow X\oplus Y$ and $i_{Y}:Y\rightarrow X\oplus Y$. To see $(ii)$ consider the inclusion maps $i_{X}:X\rightarrow X\oplus Y$ and $i_{Y}:Y\rightarrow X\oplus Y$.
Due to the hypothesis, there exist maps $f:W\rightarrow X$ and $g:W\rightarrow Y$. Thus, the compositions $i_{X}\circ f$ and $i_{Y}\circ g$ ensure the inequality $W\leq X\oplus Y$.  
Moreover, 
\begin{itemize}
\item[(iii)] if $Z$ is a common target of $X$ and $Y$, the limit of all linear maps from $X$ and $Y$ to $Z$ is a subalgebra of $X\oplus Y$;
\item[(iv)] if $W$ is a common source of $X$ and $Y$, the colimit of all linear maps from $W$ to $X$ and $Y$ is a quotient algebra of $X\oplus Y$.
\end{itemize}
both of them constituting vector spaces. 
\end{remark}


\begin{definition}
Let $A$ and $B$ be vector spaces and $I$ and $J$ be arbitrary sets. 
Consider the family of linear maps from $A\oplus B$ to all vector spaces with common sources $A$ and $B$, i.e., 
\[
\Ff_k = \{f_i: A\oplus B\rightarrow X_k \mid \text{ for all vector spaces } X_k\geq A,B \text{ and }  i\in I \}
\] 
and, dually, the family of linear maps from all vector spaces with common targets $A$ and $B$ to $A\oplus B$, i.e.,
\[
\G_k = \{g_i: Y_k\rightarrow A\oplus B \mid \text{ for all vector spaces } Y_k \leq A,B \text{ and } i\in I \}.
\]
Define $A\wedge B$ to be the kernel set $\EE $ of the equalizer of the linear maps of the family $\Ff_{k}$, $eq(\oplus_{k\in J} \Ff_{k})$, and $A\vee B$ to be the quotient set $\CC$ of the coequalizer of the linear maps of the family $\G_{k}$, $coeq(\oplus_{k\in J} \G_{k})$. These operations are well defined due to Remark \ref{exist}.
\end{definition}

\begin{figure}

\subfigure[meet operation $\wedge$]{

\begin{tikzpicture}[very thick]
  \node (a) at (0,0) {$A$};
  \node (b) at (4,0) {$B$};

  \draw[->] (0,-2) -- (a);
  \draw[->] (1,-2) -- (a);
  \draw[->] (-1,-2) -- (a);

  \draw[->] (4,-2) -- (b);
  \draw[->] (3,-2) -- (b);
  \draw[->] (5,-2) -- (b);

  \draw[<-] (0,2) -- (a);
  \draw[<-] (1,2) -- (a);
  \draw[<-] (-1,2) -- (a);

  \draw[<-] (4,2) -- (b);
  \draw[<-] (3,2) -- (b);
  \draw[<-] (5,2) -- (b);
  
 \draw[rotate around={-30:(0.5,1)},blue] (0.5,1)  ellipse (0.5 and 2);
 \draw[rotate around={30:(3.5,1)},blue] (3.5,1)  ellipse (0.5 and 2);
 \node (y) at (2,-2) {\textcolor{blue}{$A\wedge B$}};
 
 \draw[dashed,blue] (a) -- (y) -- (b);

\end{tikzpicture}
}
\subfigure[join operation $\vee $]{

\begin{tikzpicture}[very thick]
  \node (a) at (0,0) {$A$};
  \node (b) at (4,0) {$B$};

  \draw[->] (0,-2) -- (a);
  \draw[->] (1,-2) -- (a);
  \draw[->] (-1,-2) -- (a);

  \draw[->] (4,-2) -- (b);
  \draw[->] (3,-2) -- (b);
  \draw[->] (5,-2) -- (b);

  \draw[<-] (0,2) -- (a);
  \draw[<-] (1,2) -- (a);
  \draw[<-] (-1,2) -- (a);

  \draw[<-] (4,2) -- (b);
  \draw[<-] (3,2) -- (b);
  \draw[<-] (5,2) -- (b);
  
  \draw[rotate around={30:(0.5,-1)},red] (0.5,-1)  ellipse (0.5 and 2);
  \draw[rotate around={-30:(3.5,-1)},red] (3.5,-1)  ellipse (0.5 and 2);

  \node (x) at (2,2) {\textcolor{red}{$A\vee B$}};
  
   \draw[dashed,red] (a) -- (x) -- (b);

\end{tikzpicture}
}

\caption{Intuition of the defined lattice operations meet, $\wedge $ and join, $\vee $.(a) Given two elements, $A$ and $B$, the meet is defined by looking at  all the spaces which $A$ and $B$ map into to compare them. (b) For the join, we use the dual construction and compare $A$ and $B$ using all the spaces which map into $A$ and $B$.} 
\label{intuition}
\end{figure}
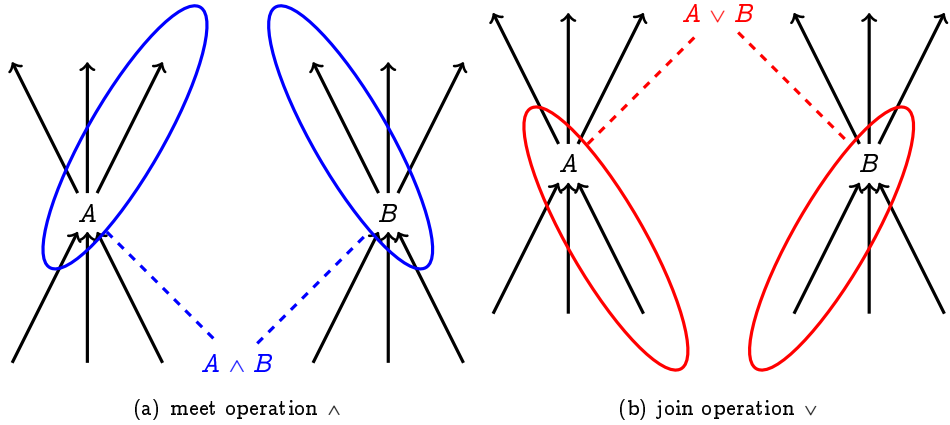

\begin{remark}
Intuitively, whenever $A$ and $B$ are vector spaces we construct $A\vee B$ as the limit of all vector spaces that have maps coming in from both $A$ and $B$ by ''gathering" together all those maps to all vector spaces $C_{i}$ with common sources $A$ and $B$: in particular, this limit is the equalizer of such maps.
Dually, we construct $A\wedge B$ as the colimit of all the linear maps from a vector space $D_{j}$ to common targets $A$ and $B$.
This intuition is represented in Figure~\ref{intuition}.
Hence, $A\wedge B$ is the limit of the $\set{A,B}$-cone and $A\vee B$ is the colimit of the $\set{A,B}$-cocone. 
Recall that (co)complete categories are the ones where the (co)limit of any diagram $F:I\rightarrow D$ exists. 
The category of vector spaces is both complete and cocomplete.
Thus, we can generalize this to an arbitrary set of vector spaces $\set{A_{0},A_{1},\dots , A_{i}, \dots }$ in the sense of complete lattices (discussed later in Section \ref{The Lattice Proofs}).
The definitions for $\wedge$ and $\vee$ have a constructive nature that will show to be useful when we later describe the computation of the operations.
To resume, given a diagram of vector spaces and linear maps $\DD$, and arbitrary vector spaces $X$ and $Y$ in $\DD$ we call \dfn{meet} of spaces $X, Y$ to the limit in $\DD$ of all linear maps from $X\oplus Y$ to common targets of $X$ and $Y$, i.e., 
    \[
  X\meet Y = \lim\set{X\to Z\from Y : 
    Z \text{ common target of $X$ and $Y$}}
  \]
Dually, we call \dfn{join} of spaces $X, Y$ to the colimit in $\DD$ of all linear maps from common sources of $X$ and $Y$ to $X \oplus Y$, i.e., 
    \[
  X\join Y = \colim\set{X\from Z\to Y : 
    Z \text{ common source of $X$ and $Y$}}
  \] 
\end{remark}

\begin{remark}
Regarding the algorithmic implementation of equalizers and coequalizers, we refer to \cite{Skr13} where, given linear maps $f$ and $g$, the authors discuss the computation of $ker(f-g)$ and $coker(f-g)$ that correspond to the computation of pullbacks and push outs, respectively. As shown above, under the assumptions of this paper, these correspond to equalizers and coequalizers. 
Furthermore, when considering families of linear maps $\FF=(f_{i})_{i\in I}$ and $\GG=(g_{j})_{j\in I}$ of more than two maps, the equalizer of $\FF$ is $\bigcap_{i,j\in I} ker(f_{i}-f_{j})$ and the coequalizer of $\GG$ is $B/\bigcup_{i,j\in I} im(g_{i}-g_{j})$.
In fact, any such solution set of multiple equations can be seen as the solution set of one equation and thus we can reduce the computation to one kernel, Dual remarks hold for the computation of the coequalizer.
\end{remark}


\subsection{The Lattice Proofs}
\label{The Lattice Proofs}
In the following result we will show that the elements of a commutative diagram of
vector spaces together with the operations $\vee $ and $\wedge $ defined above determine a lattice.  
We will refer to it as the \emph{persistence lattice} of a given diagram of vector spaces and linear maps, 
i.e., the completion of that diagram into a lattice structure using the lattice operations $\vee $ and $\wedge $.
We  shall also show the stability of the lattice operations defined above, and show that these determine a complete lattice.


\begin{theorem}
Let $\DD$ be a diagram of spaces and maps between them. 
Consider the partially ordered set $\Pe=(\DD^*;\leq)$, with the operations $\vee $ and $\wedge $ defined as above, 
where $^*$ is the closure of $P$ relative to these operations.
Then $\Pe$ constitutes a lattice.
\end{theorem}

\begin{proof}
Let us see that $A\wedge B$ is the biggest lower bound of the set $\set{A,B}$.
Due to Remark \ref{exist} we need only to see that given another vector space $D$ such that $D\leq A,B$, then there exists a linear map from $D$ to $A\wedge B$, i.e., $D \leq A\wedge B$. 
Let us consider the following diagram:

\begin{center}
\begin{tikzpicture}

\node (o) at (1,3) {$A\oplus B$};
\node (a) at (0,2) {$A$};
\node (b) at (2,2) {$B$};
\node (d) at (1,-1) {$D$};
\node (ab) at (1,0.5) {$A\wedge B$};

\draw[arrows=-latex'] (a) -- (o) node[pos=.5,right] {};
\draw[arrows=-latex'] (b) -- (o) node[pos=.5,right] {}; 
\draw[->, bend left] (d) to node {} (a) ;
\draw[->, bend right] (d) to node {} (b) ;
\draw[arrows=-latex', dashed] (ab) -- (a) node[pos=.5,right] {};
\draw[arrows=-latex', dashed] (ab) -- (b) node[pos=.5,right] {};
\draw[arrows=-latex', dashed,blue] (d) -- (ab) node[pos=.5,right] {};

\end{tikzpicture}  
\end{center}  

The compositions of either with the maps from $A$ and $B$ to some common target $C$ ($A\oplus B$, for instance) commute by assumption.
Due to the construction of $A\wedge B$ as a limit, we get that $D\leq A\wedge B$ by universality.
Hence, $A\wedge B$ is the greatest lower bound (the biggest subalgebra) regarding all the other subalgebras of $A\oplus B$ that are maps from $A\oplus B$ to the vector spaces above both $A$ and $B$.
The proof that $A\vee B$ is the least upper bound (the finest partition) of the set $\set{A,B}$ is analogous and derives from the universality of its construction as a colimit.
\end{proof}


\begin{theorem}\label{stabilize}
Given vector spaces $A$ and $B$, the construction of $A\wedge B$ and $A\vee B$ stabilizes.
\end{theorem}

\begin{proof}
In the following proof we will show that the skew lattice construction stabilizes, i.e., whenever we are given vector spaces $A$ and $B$ and
\begin{itemize}
\item[(1)] we first construct $A\wedge B$ from $A,B\leq A\oplus B$, 
\item[(2)] then we construct $A\vee B$ from $A\wedge B\leq A,B$, 
\item[(3)] then we again construct $(A\wedge B)'$ from $A,B\leq A\vee B$,
\end{itemize}

\noindent we can ensure that $(A\wedge B)'=A\wedge B$.
The dual result follows analogously.

\textbf{Case 1: Sources.}
In this case, we assume that the elements are two sources and that there exists an element above both of them. We denote the elements $A$, $B$ and $C$, respectively. We are then able to define $M=A\meet B$ that is constituted by elements $(a,b)$ of $A\oplus B$ such that $(f,0)(a,b) = (g,0)(a,b)$, where $f$ and $g$ map to $C$. Since there is now an element below $A$ and $B$, we can define  $J=A\join B$ as all the quotient space of $A\oplus B$. Define $M\rto A\oplus B$  where the map is $(k,\ell)$. Therefore we now have $A\oplus B \rto A\oplus B /\langle (k(x), \ell(x)) \mid x\in M \rangle$. Call these maps $v$ and $w$. What remains to show is that the elements which satisfy  $(v,0)(a,b) = (0,w)(a,b)$ are the same as above. Now if $(f,0)(a,b) = (g,0)(a,b)\neq (0,0)$, by commutivity and universality,  $(v,0)(a,b) = (0,w)(a,b)\neq (0,0)$. However, if $(f,0)(a,b) = (g,0)(a,b) =  (0,0)$, then there exists an element $m \in M$ such that $m\mapsto (a,b)$ which implies that $(v,0)(a,b) = (0,w)(a,b)$, since this is precisely the relation in the definition.  Since $M$ can only get smaller with additional constraints, it follows that the resulting $M$ has stabilized.

\textbf{Case 2: Targets.}
In this case, we assume that the elements are two sources and that there exists an element below them. We denote the elements $A$, $B$ and $C$ respectively. We define  $J=A\join B$, constituted by the quotient  $A\oplus B / \langle (f(x), g(x))\mid x\in C \rangle$. Denote this map $(k,\ell)$. Based on this we define the $M=A\meet B$ as the subspace such that $(k,0)(a,b) =(0,\ell)(a,b)$. Denote the map from this space to the direct sum as $(v,w)$. Now we need to show $A\oplus B /\langle (f(c), g(c))\mid c\in C \rangle = A\oplus B /\langle (v(m), w(m))\mid m\in M \rangle$. By universality  it follows that there exists an $m \in M $ such that $c \mapsto m$ and hence $f(c) = v(m)$ and $g(c) = w(m)$. It follows that $f(c) \theta g(c)$ is equivalent to  $v(m) \theta w(m)$. 
If we do not want to use universality, if $(f,g)(c)\neq (0,0)$, there must be an element in $J$ such that $k((f(c)) = \ell(g(c))=j$. Hence we conclude that there is an element $c\mapsto m$. If $(f,g)(c) =  (x,0)$, then  by the quotient $k(f(c)) = 0 $ and again there must be an element $m \mapsto (x,0)$.  Finally if $(f,g)(c) =  (0,0)$, there is no element other than 0 such that $k(f(c)) = \ell(g(c))$ and hence $c \mapsto (0,0) \in M$. 

\end{proof}


\begin{theorem}\label{complete}
Persistence lattices are complete, i.e., both of the lattice operations extend to arbitrary
joins $\bigjoin_i D_i$ and meets $\bigmeet_i D_i$ (note that both
$\bigjoin_i D_i$ and $\bigmeet_i D_i$ might not be in $\DD$).
\end{theorem}

\begin{proof}
Consider a subset $S$ of the underlying set of spaces of the given persistence lattice $\Pe$. 
Take their direct sum $X=\oplus_\ell \set{A_{\ell}\in S}$.
To see that the arbitrary set $S$ has a general meet just consider $\bigwedge S$ to be the limit of all the maps from all vector spaces $A_{\ell}\in S$ to a common vector space $\oplus_{k} C_{k}$ such that $A,B\leq C_{k}$, for each $k$, i.e., 
\[
\bigwedge S = \set{x\in X: f_{i}(x)=f_{j}(x) \text{,  for all  } f_{i},f_{j}\in \bigcup_{k}Hom(X,C_{k})}.
\]
\noindent This is the kernel set determined by the parcels of the direct sum $X$ that satisfy the system of equations determined by the considered maps, i.e.,
\[
\bigwedge_{\ell} A_{\ell}= \set{x\in \oplus_{\ell} A_{\ell}: f_{A_{i}A_{j}}(x)=f_{A_{u}A_{v}}(x)}.
\]
Dually, $\bigvee S$ is the colimit of the union of all maps from a common vector space $\oplus_{k} D_{k}$ all vector spaces $A_{i}\in S$ such that $D_{k}\leq A,B$, for each $k\in I$.
Hence, 
\[
\bigvee_{\ell} A_{\ell}=(\oplus_{\ell} A_{\ell})/\langle (f_{i}(x),f_{j}(x))\mid x\in \oplus_{k} D_{k}\rangle
\]
which is the quotient of the product of the vector spaces $A_{\ell}$ by the equivalence generated by the union of respective equivalences, i.e., 
\[
\bigvee_{\ell} A_{\ell}=(\oplus_{\ell} A_{\ell})/\langle \bigcup \theta_{A_{i}A_{j}}\rangle.
\]
\end{proof}

\begin{remark}
According to our definition of $\meet$ and $join$, 
\begin{itemize}
\item[(i)] the $\bigmeet$ of spaces $X_i$ is the limit in $\Pe$ of all linear maps from $\oplus_{i\in I} X$ to common targets of $X_i$, i.e., 
    \[
  \bigmeet_{i\in I} X_i = \lim\set{X_i\to Z :  Z \text{ common target of $X_i$}}
  \]
\item[(i)]  the $\bigjoin$ of spaces $X_i$ is the colimit in $\Pe$ of all linear maps from common sources of $X_i$ to $\oplus_{i\in I} X$ , i.e., 
    \[
  \bigjoin_{i\in I} X_i = \colim\set{X_i\from Z :  Z \text{ common source of $X_i$}}
  \]
\end{itemize} 
\end{remark}

\begin{remark}
Completeness is a very important property in the study of ordered structures.
The open sets of a topological space, ordered by inclusion, are examples of such structures where $\vee$ is given by the union of open sets and $\wedge$ by the interior of the intersection. 
In the last section we will see an algorithm application for this particular lattice property. 
We will refer to it as \emph{the largest injective} by then.  
\end{remark}


\subsection{The Lattice Properties}

In the following we describe some of the most relevant characteristics of the lattice that we have described in the earlier section.
We shall see that, besides the algebraic properties due to its lattice nature, it is also modular and distributive. 
%
%

\begin{remark}
Let us first have a look at the properties of the operations $\wedge$ and $\vee$ of the persistence lattice $\HH$ constructed above over an input poset. 
The identity map implies that $A\wedge A=A$ and $A\vee A=A$. 
This algebraic property follows from the order structure of the correspondent persistence lattice. 
The equivalence between the algebraic structure and the order structure of the underlying algebra ensures that a linear map $f:A\rightarrow B$ exists iff $A=A\wedge B$ iff $A\vee B=B$. 
Moreover, the following lattice identities hold:
\[
A\wedge (A\vee B)=A=A\vee (A\wedge B)=A.
\]


\end{remark}


The following result will enlighten this theory with a nice relation between the lattice operations and the direct sum. This property is not frequently used in the study of lattice properties but will permit us to show the distributivity of a persistence lattice in the next paragraphs.

\begin{theorem}\label{ses}
Let $A$ and $B$ be vector spaces. Then, \[
A\wedge B\rightarrow A\oplus B\rightarrow A\vee B \text{  is a short exact sequence.}
\]
\end{theorem}

\begin{proof}
First observe that the limit map $f:A\wedge B\rightarrow A\oplus B$ is injective and the colimit map $g:A\oplus B\rightarrow A\vee B$ is surjective (cf. \cite{La98}).
We thus need to show that $\im f=\ker g$ to prove the isomorphism 
\[
A\vee B\cong A\oplus B/f(A\wedge B).
\]

If $y\in \im f$ then there exists $x\in A\wedge B$ mapping to $y$ such that $g_{i}(x)=g_{j}(x)$ for all $g_{k}:A\oplus B\rightarrow A\vee B$ and thus $y\in \ker g$.
On the other hand, if $x\in \ker g$, then $g_{|A}(x_{|A}) = g_{|B}(x_{|B})$ implying there exists an element in $x\in A\wedge B$ which maps to $y$.

\end{proof}


\begin{theorem}\label{distributive}
Persistence lattices are distributive.
\end{theorem}

\begin{proof}
Let $A$, $B$ and $X$ be vector spaces such that $X\vee A=X\vee B$ and $X\wedge A=X\wedge B$ in order to show that $A\cong B$. 
Consider the following commutative diagram of spaces:

\begin{center}
\begin{tikzpicture}[scale=.7]

  \node (c) at (0,2.5) {$ X\vee A=X\vee B$};
  \node (d) at (0,-2.5) {$X\wedge A=X\wedge B$};
  \node (a) at (-2.5,0) {$A$};
  \node (b) at (2.5,0) {$X$};
  \node (x) at (0,0) {$B$};
  
\draw[arrows=-latex'] (a) -- (c) node[pos=.5,left] {$f$};
\draw[arrows=-latex'] (d) -- (a) node[pos=.5,left] {$g$};
\draw[arrows=-latex'](x) -- (c) node[pos=.5,left] {$u$};
\draw[arrows=-latex'] (d) -- (b) node[pos=.5,right] {$t$};
\draw[arrows=-latex'](d) -- (x) node[pos=.5,left] {$v$};
\draw[arrows=-latex'](b) -- (c) node[pos=.5,right] {$s$};

\end{tikzpicture}
\end{center}

\noindent
The result will follow from the definition of distributivity for the lattice operations, the Five Lemma
and exactness of the sequence (cf. Theorem \ref{ses})
\begin{equation*}
0\rightarrow Y\wedge Z \xrightarrow{f} Y\oplus Z \xrightarrow{g} Y\vee Z \rightarrow 0 
\end{equation*}
Consider the the following diagram
\begin{center}
\begin{tikzcd}
0 \arrow{r} \arrow[leftrightarrow]{d}{\cong} & A\wedge X \arrow{r} \arrow[leftrightarrow]{d}{\cong} & A\oplus X \arrow{r} &A\vee X \arrow{r} \arrow[leftrightarrow]{d}{\cong}&0 \arrow[leftrightarrow]{d}{\cong} \\
0 \arrow{r} & B\wedge X \arrow{r} & B\oplus X \arrow{r} &B\vee X \arrow{r} &0 
\end{tikzcd}
\end{center}
The first and last isomorphism are trivial, while the other isomorphisms follow by assumption. 
The existence of the linear map $f: A\oplus B \rightarrow B\oplus X$ is ensured by the fact that we are dealing with vector spaces, assuming the commutativity of the diagram. 
Therefore, by the Five Lemma, we conclude that $A\oplus X \cong B\oplus X$ and hence $A\cong B$, concluding the proof.

\end{proof}

The distributive property is of great interest in the study of order structures. With it we are able to retrieve a rich structure satisfying many interesting identities. The next result follows directly from the distributivity of persistence lattices.

\begin{corollary}
The persistence lattice intervals $[A \wedge B, B]$ and $[A, A \vee B]$ are isomorphic due to the maps $f: [A\wedge B,B] \rightarrow [A,A\vee B]$, defined by $X\mapsto X\vee A$, and $g:  [A, A\vee B] \rightarrow [A\wedge B,B]$, defined by $Y\mapsto Y\wedge B$.
\end{corollary}

\begin{remark}
Due to Dilworth's results on poset decompositions, there exists an antitotal order of vector spaces $S$ and a partition of the order in $A$ into a family $F$ of total orders of vector spaces such that the number of total orders in the partition equals the cardinality of $S$ and, thus, $S$ is the largest antitotal order in the order, and $F$ must be the smallest family of total orders into which the order can be partitioned. 
Dually, the size of the largest total order of vector spaces in a finite poset of vector spaces as such equals the smallest number of antitotal orders of vector spaces into which the order may be partitioned.
\end{remark}


\begin{theorem}\label{discrete}
Persistence lattices are discrete, finite and bounded.
\end{theorem}

\begin{proof}
In the following we will give an upper bound for the number of elements of a persistence lattice of a given diagram of spaces.
The finiteness of the lattice implies that it is discrete and complete. Thus, it follows that it is a bounded lattice. 
Indeed, an upper bound for the number of elements of the persistence lattice correspondent to a diagram with $\mid V\mid =n$ is given by 
\[
\sum_{i} \binom{n}{i} 2^{i-1}\leq 2^{n}.2^{n}=2^{2n}.
\]
To see the above bound consider a string of $V_i$'s. Since the
operations are commutative and associative, we will need to only
consider all combinations of nodes which are included in the string.
To get an element of the lattice, we must also consider the two operations. For a string of length of $m$, this implies $m-1$ operations. Since we have two operations this implies there are $2^{(m-1)}$ operations on the string. Since $m<n$, we can bound the sum by $2^{2n}$, implying that we add a  finite number of elements. 
\end{proof}

\begin{remark}
This is a very loose bound intended only to illustrate finiteness. In practice, there will be far fewer elements due to distributivity and even fewer elements of interest.  
\end{remark}



\begin{theorem}
Persistence lattices constitute complete Heyting algebras.
\end{theorem}

\begin{proof}
Recall that nonempty finite distributive lattices are bounded and complete, thus forming Heyting algebras. 
Hence, this result follows from Theorems \ref{distributive}, \ref{complete} and \ref{discrete}.
\end{proof}

\begin{remark}
Whenever $A$ and $B$ are vector spaces in a diagram, there exists a vector space $X$ that is maximal in the sense of $X\wedge A\leq B$, i.e., the implication operation is given by the colimit 
\[
A \Rightarrow B=\bigvee \set{X_i\in L \mid \bigoplus_i (X_i\wedge A) \rightarrow B}.
\]
Observe that the case of standard persistence we have that 
\[
A\Rightarrow B = \begin{cases} B, & \mbox{if } B\leq A\\ 1, & \mbox{if } A\leq B \end{cases}.
\]
The study of the interpretation of the implication operation in the framework of other general models of persistence, as zig-zag or multidimensional persistence, is a matter of further research.
\end{remark}

\begin{remark}
Persistence lattices $\Pe$ are not Boolean algebras. 
To see this just consider the standard persistence case that is represented by a total order, or the total order $\set{C,B,D}$ in the above bifiltration and observe that there is no $X\in L$ such that $B\wedge X=D$ and $B\vee X=C$. Hence, $B$ also doesn't have a complement in $\Pe$.
\end{remark}

\begin{remark}
The results of this section permit us to discuss several directions of future work that can contribute with further information on the order and algebraic properties of this structure and motivate the construction of new algorithms. 
A \emph{topos} is essentially a category that ''behaves" like a category of sheaves of sets on a topological space, while \emph{sheaves} of sets are functors designed to track locally defined data attached to the open sets of a topological space and transpose it to a global perspective using a certain ''gluing property". 
Topos theory has important applications in algebraic geometry and logic (cf. \cite{Lang} and \cite{Joh86}), and has recently been used to construct the foundations of quantum theory (cf. \cite{Is08}).
The category of sheaves on a Heyting algebra is a topos (cf. \cite{CatCS}).
Whenever skew lattices, a noncommutative variation of lattices, satisfy a certain distributivity, they constitute sheaves over distributive lattices (and over Heyting algebras in particular (cf. \cite{Ba13}). 
The study of such algebras, developed by the second author of this paper in \cite{Le13}, might be of great interest to the research on the properties of persistence lattices and their interpretation in the framework of persistent homology.
Furthermore, complete Heyting algebras are of great importance to study of frames and locales that form the foundation of pointless topology, leading to the categorification of some ideas of general topology (cf. \cite{Joh86}).
\end{remark}

\begin{remark}
A natural and well studied relationship between lattice theory and topology is described by the duality theory \cite{Da02}.  
These dualities are of great interest to the study of algebraic and topological problems taking advantage of the categorical equivalence between respective structures (cf. \cite{Fa96}).
In the case of complete Heyting algebras, the Esakia duality permits the correspondence of such algebras to dual spaces, called \emph{Esakia spaces} that are compact topological spaces equipped with a partial order, satisfying a certain separation property that will imply them to be Hausdorff and zero dimensional (cf. \cite{Be06}).  
These spaces are a particular case of \emph{Priestley spaces} that are homeomorphic to the spectrum of a ring (cf.\cite{Be10}). 
We are interested in the study of such topological spaces and correspondent ring. 
\end{remark}


\section{Algorithms and Applications}
\label{Algorithms and Applications}

We now give some interpretations of both the order structure and the algebraic structure of the lattice in the framework of persistent homology. 


\subsection{Interpretations Under Persistence}

We saw that in the case of standard persistence, we have a total order where $A$ and $B$ are related and thus $(L1)$ tells us that, $\X_{m}\wedge \X_{n}=\X_{m}$, the domain of the map $f$ connecting $\X_{m}$ and $\X_{n}$, while $\X_{m}\vee \X_{n}=\X_{n}$, its codomain.
On the other hand, to analyze the multidimensional case we saw that using  
\[
\X_{n,m}\wedge \X_{p,q}=\X_{\min\set{n,p},\max\set{m,q}} \text{  andÊ } \X_{n,m}\vee \X_{p,q}=\X_{\max\set{n,p},\min\set{m,q}}.
\] 
for the meet and join respectively we recover the rank invariant. 
We will return to the bifiltration case but first discuss its connections with zig-zag persistence. 
In the case of zig-zag persistence, we get the following diagram:

\begin{center}
\begin{tikzpicture}
\node (a) at (0,0 ) {$\Hg(\X_0)$};
\node (b) at (2,0 ) {$\Hg(\X_2)$};
\node (c) at (4,0 ) {$\Hg(\X_4)$};
\node (d) at (6,0 ) {$\Hg(\X_6)$};
\node (e) at (8,0 ) {$\Hg(\X_8)$};
\node (w) at (1,1 ) {$\Hg(\X_1)$};
\node (x) at (3,1 ) {$\Hg(\X_3)$};
\node (y) at (5,1 ) {$\Hg(\X_5)$};
\node (z) at (7,1 ) {$\Hg(\X_7)$};
\draw[->] (a) -- (w);
\draw[->] (b) -- (w);
\draw[->] (b) -- (x);
\draw[->] (c) -- (x);
\draw[->] (c) -- (y);
\draw[->] (d) -- (y);
\draw[->] (d) -- (z);
\draw[->] (e) -- (z);
\end{tikzpicture}
\end{center}
Without loss of generality, if we assume that we have an alternating zig-zag as above, we see that we have a partial order: the odds are strictly greater than the even indexed spaces.  This is not an interesting partial order as most elements are incomparable. 
In ~\cite{Car09} and~\cite{Be12} it was noted that using unions and relative homology, the above could be extended to a case where all elements become comparable with possible dimension shifts. 
The resulting zig-zag can be extended into a M\" obius strip through exact squares. By exactness any two elements can be compared by considering unions and relative homologies as shown in Figure~\ref{fig:zz}. 

\begin{figure}[ht]
\begin{center}
\includegraphics[height=5cm]{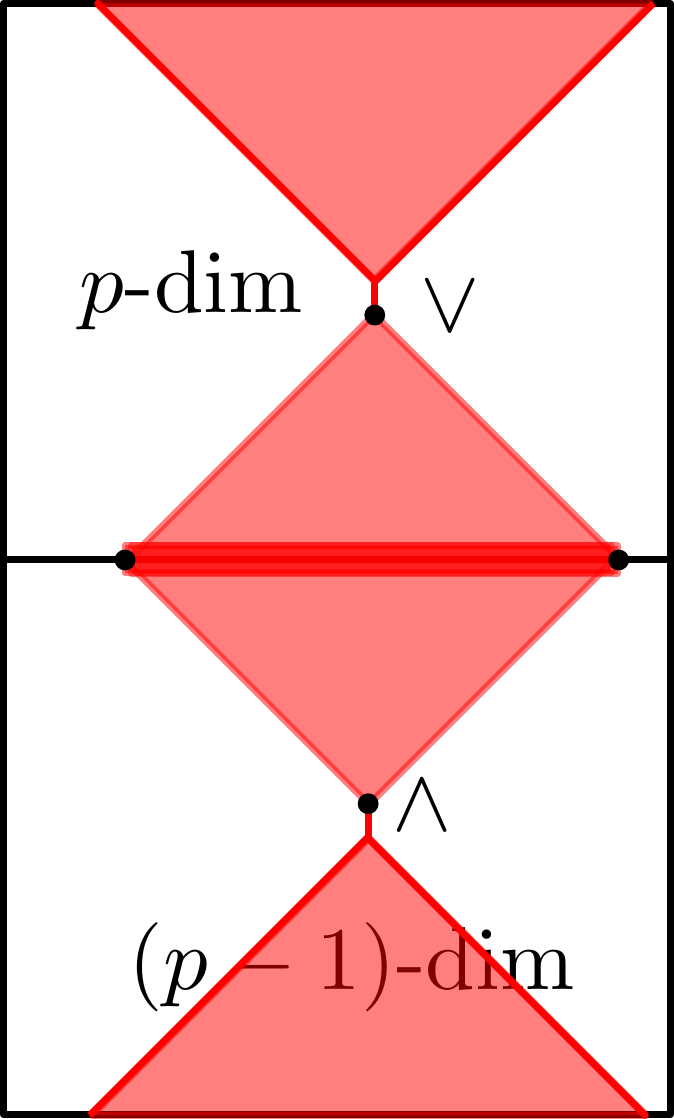}
\caption{\label{fig:zz}  Here we show a possible choice of meet and
  join for  zig-zag persistence based on the M\" obius strip
  construction of~\cite{Car09}.}
\end{center}
\end{figure}

Using a special case of our construction, using pullbacks and pushouts as limits and colimits, the authors in~\cite{Sk12}, developed a parallelized algorithm for computing zig-zag persistence.  

To compare two general elements define
\begin{equation*}
\Hg_*(\X_i)\meet\Hg_*(\X_j) = \begin{cases} K \rto \Hg_*(\X_i)\oplus \Hg_*(\X_j)\rightrightarrows \Hg_*(\X_{i+1}) \qquad  j=i+2\\
\Hg_*(\X_i)\meet\Hg_*(\X_{i+2})\meet\cdots\meet\Hg_*(\X_j) \
\end{cases}
\end{equation*}
and
\begin{equation*}
\Hg_*(\X_i)\join\Hg_*(\X_j) = \begin{cases} \Hg_*(\X_{i+1}) \rightrightarrows  \Hg_*(\X_i)\oplus \Hg_*(\X_j) \rto P \qquad  j=i+2\\
\Hg_*(\X_i)\join\Hg_*(\X_{i+2})\join\cdots\join\Hg_*(\X_j) \
\end{cases}
\end{equation*}
With this definition it is not difficult to verify the following results
\begin{enumerate}
\item The rank of $\Hg_*(\X_i)\meet\Hg(\X_j)  \rightarrow
  \Hg_*(\X_i)\join\Hg(\X_j)$ is equal to the  rank in the original zig-zag
  definition.
\item The structure can be built up iteratively, comparing all
  elements two steps away then three steps away and
  so on, leading to the parallelized algorithm. 
\end{enumerate}
\begin{remark}
In \cite{Sk12}, an additional trick was used so that only the meets had to
be computed. 
\end{remark}


\subsection{Largest Injective}

%
%
%
%
%
%

For the first application, we consider the computation of the largest injective of a diagram. 
In principle, we are looking for something which persists over an entire diagram. 
While satisfying the properties of the underlying lattice structure, the largest injective must fulfill to be in the following images 
\begin{equation*}
\im\left(\Hg_*(\X_i) \meet \Hg_*(\X_j) \rto \Hg_*(\X_i) \join
\Hg_*(\X_j)\right) \qquad \forall i,j
\end{equation*}
By completeness, it follows that this can be written as 
\begin{equation*}
\im \left(\bigwedge\limits_i \Hg_*(\X_j) \rto \bigvee\limits_i
\Hg_*(\X_i)\right) 
\end{equation*}
Using the order structure, we
can rewrite the above as
\begin{equation*}
\im \left(\bigwedge\limits_{i\in\mathrm{sources}} \Hg_*(\X_j) \rto \bigvee\limits_{j\in\mathrm{targets}}
\Hg_*(\X_j)\right). 
\end{equation*}
Recall that \emph{sources} are all the elements in original diagram which are not the codomain of any maps and \emph{targets} are the elements which are not the domain of any maps. 
Assuming we have $n$ sources, $m$ targets and the longest total order in the diagram is $k$ assuming an $O(1)$ time to compute a $\join$ or $\meet$ of two elements, we have a run time of $O(n+m+k)$. 
On a parallel machine, the operations can be computed independently and using associativity, we can construct the total meet/join using a binary tree scheme, giving a run time of $O(k+\log (\max(n,m)) )$.

\begin{center}

     \begin{tikzpicture}[scale=.8] 

       \node (x1) at (1,1) {$\X_{1}$};
        \node (x2) at (3,1) {$\X_{2}$};
        \node (x3) at (5,1) {$\X_{3}$};
        
       \node (x4) at (1,-2) {$\X_{4}$};
        \node (x5) at (3,-2) {$\X_{5}$};
        \node (x6) at (5,-2) {$\X_{6}$};
        
        \draw[fill=black] (0,0) circle (0.05cm);
        \draw[fill=black] (2,0) circle (0.05cm);
        \draw[fill=black] (4,0) circle (0.05cm);
        \draw[fill=black] (6,0) circle (0.05cm);

        \draw[fill=black] (0,-1) circle (0.05cm);
        \draw[fill=black] (2,-1) circle (0.05cm);
        \draw[fill=black] (4,-1) circle (0.05cm);
        \draw[fill=black] (6,-1) circle (0.05cm);
        
        \draw[->] (0,0) edge (x1);
        \draw[->] (2,0) edge (x1);
        \draw[->] (2,0) edge (x2);
        \draw[->] (4,0) edge (x2);
        \draw[->] (4,0) edge (x3);
        \draw[->] (6,0) edge (x3);
        
        \draw[->] (x4) edge (0,-1);
        \draw[->] (x4) edge (2,-1);
        \draw[->] (x5) edge (2,-1);
        \draw[->] (x5) edge (4,-1);
        \draw[->] (x6) edge (4,-1);
        \draw[->] (x6) edge (6,-1);
        
        \node (x7) at (2,-3) {$\X_{7}$};
        \node (x8) at (4,-3) {$\X_{8}$};
        \node (x9) at (3,-4) {$\X_{9}$};

        \draw[->,blue] (x7) edge (x4);
        \draw[->,blue] (x7) edge (x5);
        \draw[->,blue] (x8) edge (x5);
        \draw[->,blue] (x8) edge (x6);
        \draw[->,blue] (x9) edge (x7);
        \draw[->,blue] (x9) edge (x8);       
        
        \node (x10) at (3,2) {$\X_{10}$};

        \draw[->, dashed, blue] (x1) edge (x10);
        \draw[->, dashed, blue] (x2) edge (x10);
        \draw[->, dashed, blue] (x3) edge (x10);

     \end{tikzpicture}

\begin{equation*}
\im \left(\bigwedge\limits_{i\in\mathrm{sources}} \Hg_*(\X_j) \rto \bigvee\limits_{j\in\mathrm{targets}}\Hg_*(\X_j)\right) 
\end{equation*}

\end{center}

Unfortunately, we cannot always compute the meet or join in constant time as we may need to compose a linear number of maps. In the future, we will do a more fine grain analysis, but we note that given that we have a
distributive lattice, all maximal total orders are of constant length, allowing us to bound the time to compute any meet and join by this length.


\subsection{Stability of the Lattice}

Here we look at a possible description of stability relating to a persistence lattice. 
The general idea is to show that if some local conditions hold, we can infer the existence of some persistent classes. 

\begin{lemma}\label{break}
Let $A$, $B$, $C$ and $D$ be vector spaces such that $A\wedge B\leq C$ and $D\leq A\vee B$. 
Then, $A\vee B \leq C$ and $D\leq A\wedge B$. 
\end{lemma}

\begin{proof}
Assume the existence of a linear map $f:A\wedge B\rightarrow C$
As $A\wedge B$ is a subalgebra of $A\oplus B$ then it is possible to construct linear maps $f_{A}:A\rightarrow C$ and $f_{B}:B\rightarrow C$ implying that $A,B\leq C$. 
Thus, the universality of $A\vee B$ constructed as a coequalizer implies the existence of a unique linear map $h:A\vee B\rightarrow C$, i.e., $A\vee B\leq C$.

\begin{center}
\begin{tikzpicture}[scale=.5]

  \node (c) at (0,2) {$A\vee B$};
  \node (d) at (0,-2) {$A\wedge B$};
  \node (a) at (-2,0) {$A$};
  \node (b) at (2,0) {$B$};
  \node (v) at (0,4) {$C$};

\draw[arrows=-latex',dashed, blue] (d) -- (c) node[pos=.5,above] {};
\draw[arrows=-latex',dashed, blue] (c) -- (v) node[pos=.5,above] {};
  
\draw[arrows=-latex'] (a) -- (c) node[pos=.5,above] {};
\draw[arrows=-latex'](b) -- (c) node[pos=.5,left] {};
\draw[arrows=-latex'] (d) -- (a) node[pos=.5,above] {};
\draw[arrows=-latex'](d) -- (b) node[pos=.5,right] {};

\end{tikzpicture}
\end{center}

Dually, the existence of a linear map $g:D\rightarrow A\vee B$ implies that $D\leq A,B$ so that the universality of $A\wedge B$ as an equalizer implies the existence of a linear map $k:D\rightarrow A\wedge B$, i.e., $D\leq A\wedge B$.
\end{proof}

We can now state the following theorem:
\begin{theorem}
Let $A$, $B$, $C$ and $D$ be vector spaces such that $A\leq B$ and $C\leq D$.
Then, $A\vee C\leq B\wedge D$. 
\end{theorem}

\begin{proof}
Assume that $A\leq B$ and $C\leq D$ and consider the following diagram:

\begin{center}
\begin{tikzpicture}[scale=.5]

  \node (c) at (0,2) {$A\vee C$};
  \node (d) at (0,-2) {$A\wedge C$};
  \node (a) at (-2,0) {$A$};
  \node (b) at (2,0) {$C$};
  \node (v) at (0,4) {$B\wedge D$};
    \node (t) at (-2,6) {$B$};
  \node (r) at (2,6) {$D$};
    \node (u) at (0,8) {$B\vee D$};

\draw[arrows=-latex',dashed, blue] (d) -- (c) node[pos=.5,above] {};
\draw[arrows=-latex',dashed, blue] (c) -- (v) node[pos=.5,above] {};
\draw[arrows=-latex',dashed, blue] (v) -- (u) node[pos=.5,above] {};
  
\draw[arrows=-latex'] (a) -- (c) node[pos=.5,above] {};
\draw[arrows=-latex'](b) -- (c) node[pos=.5,left] {};
\draw[arrows=-latex'] (d) -- (a) node[pos=.5,above] {};
\draw[arrows=-latex'](d) -- (b) node[pos=.5,right] {};

\draw[arrows=-latex'] (a) -- (t) node[pos=.5,above] {};
\draw[arrows=-latex'](b) -- (r) node[pos=.5,left] {};

\draw[arrows=-latex'] (v) -- (t) node[pos=.5,above] {};
\draw[arrows=-latex'](v) -- (r) node[pos=.5,left] {};
\draw[arrows=-latex'] (t) -- (u) node[pos=.5,above] {};
\draw[arrows=-latex'](r) -- (u) node[pos=.5,left] {};

\end{tikzpicture}
\end{center}

As $A\wedge C\leq B\vee D$, Lemma \ref{break} implies that the map $f:A\wedge C\rightarrow B\vee D$ decomposes into maps 
\[
A\wedge C\rightarrow A\vee C\rightarrow B\wedge D\rightarrow B\vee D.
\]
\end{proof}

To place this into context, consider $A\rightarrow B$ to be part of one filtration and $C\rightarrow D$ a second filtration such that they are interleaved. 
In this case for any class in $A\rightarrow B$, $A\wedge C$, and $B\vee D$ must also be in $C\rightarrow D$. In this case, the idea is that local conditions such as $A\wedge C \rightarrow A\vee C$ and $B\wedge D\rightarrow B\vee D$, imply something about the persistence between other elements. In above case, if we assume $\epsilon$-interleaving we can recover such a statement on these local conditions.  
We now give a more general statement:
\begin{theorem}
Let $A$, $B$, $C$ and $D$ be vector spaces.
Then $(A\wedge C)\vee (B\wedge D)\leq (A\vee C)\wedge (B\vee D)$.
\end{theorem}

\begin{proof}
Consider the diagram of Figure~\ref{stabfig} where $R_{1}=A\vee C$, $R_{2}=B\vee D$, $P_{1}=A\wedge C$ and $P_{2}=B\wedge D$.
The existence of the dashed maps is guaranteed by Lemma~\ref{break} and the fact that $A\wedge B\wedge C\wedge D\leq A\vee B\vee C\vee D$.

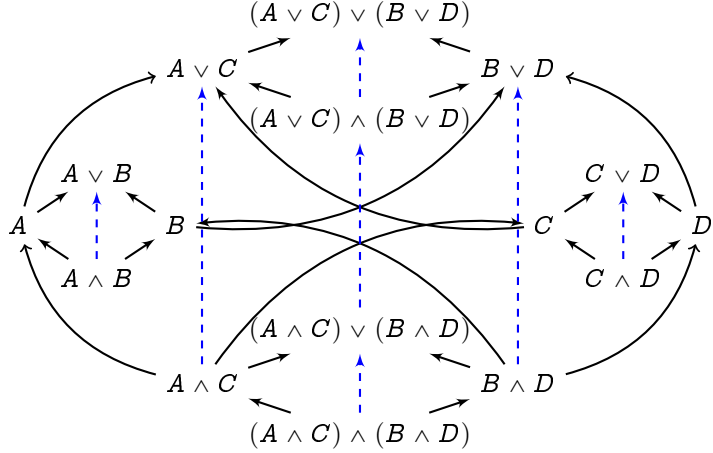
\begin{figure}

\begin{center}
\begin{tikzpicture}[thick, scale=.7]
  
  \node (1) at (0,4) {$(A\vee C)\vee (B\vee D)$};
  \node (2) at (-3,3) {$A\vee C$};
  \node (3) at (3,3) {$B\vee D$};
  \node (4) at (0,2) {$(A\vee C)\wedge (B\vee D)$};
  
  \node (5) at (-5,1) {$A\vee B$};
  \node (7) at (-6.5,0) {$A$};
  \node (8) at (-3.5,0) {$B$};
  \node (11) at (-5,-1) {$A\wedge B$};
  
  \node (6) at (5,1) {$C\vee D$};  
  \node (9) at (3.5,0) {$C$}; 
  \node (10) at (6.5,0) {$D$};  
  \node (12) at (5,-1) {$C\wedge D$}; 
  
  \node (13) at (0,-2) {$(A\wedge C)\vee (B\wedge D)$};  
  \node (14) at (-3,-3) {$A\wedge C$}; 
  \node (15) at (3,-3) {$B\wedge D$};  
  \node (16) at (0,-4) {$(A\wedge C)\wedge (B\wedge D)$};

\draw[arrows=-latex'] (7) -- (5) node[pos=.5,left] {};
\draw[arrows=-latex'] (8) -- (5) node[pos=.5,left] {};
\draw[arrows=-latex'] (11) -- (7) node[pos=.5,left] {};
\draw[arrows=-latex'] (11) -- (8) node[pos=.5,left] {};
\draw[arrows=-latex'] (9) -- (6) node[pos=.5,left] {};
\draw[arrows=-latex'] (10) -- (6) node[pos=.5,left] {};
\draw[arrows=-latex'] (12) -- (9) node[pos=.5,left] {};
\draw[arrows=-latex'] (12) -- (10) node[pos=.5,left] {};

\draw[arrows=-latex'] (2) -- (1) node[pos=.5,left] {};
\draw[arrows=-latex'] (3) -- (1) node[pos=.5,left] {};
\draw[arrows=-latex'] (4) -- (2) node[pos=.5,left] {};
\draw[arrows=-latex'] (4) -- (3) node[pos=.5,left] {};
\draw[->, bend left] (7) to node {} (2) ;
\draw[arrows=-latex', bend right] (8) to node {}  (3);
\draw[arrows=-latex', bend left] (9) to node {}  (2) ;
\draw[->, bend right] (10) to node {} (3) ;
\draw[arrows=-latex', bend left] (14) to node {} (9) ;
\draw[->, bend left] (14) to node {} (7) ;
\draw[arrows=-latex', bend right] (15) to node {} (8) ;
\draw[->, bend right] (15) to node {} (10) ;
\draw[arrows=-latex'] (14) -- (13) node[pos=.5,left] {};
\draw[arrows=-latex'] (15) -- (13) node[pos=.5,left] {};
\draw[arrows=-latex'] (16) -- (14) node[pos=.5,left] {};
\draw[arrows=-latex'] (16) -- (15) node[pos=.5,left] {};

\draw[arrows=-latex',dashed, blue] (16) -- (13) node[pos=.5,above] {};
\draw[arrows=-latex',dashed, blue] (13) -- (4) node[pos=.5,above] {};
\draw[arrows=-latex',dashed, blue] (4) -- (1) node[pos=.5,above] {};
\draw[arrows=-latex',dashed, blue] (11) -- (5) node[pos=.5,above] {};
\draw[arrows=-latex',dashed, blue] (12) -- (6) node[pos=.5,above] {};
\draw[arrows=-latex',dashed, blue] (14) -- (2) node[pos=.5,above] {};
\draw[arrows=-latex',dashed, blue] (15) -- (3) node[pos=.5,above] {};

\end{tikzpicture}
\end{center}

\caption{Hasse diagram representation of the stabilization theorem for subalgebras of a persistence lattice.}
\label{stabfig}
\end{figure}
\end{proof}

Here we do not introduce the notion of metrics or interleaving to give a more substantial result. 
However, we believe such a result is possible and we will address it in further work.


\subsection{Sections}

Finally we return to the bifiltration case to highlight the difference between our construction and the one we presented in Section~\ref{Problem Statement} which yielded the rank invariant. 
Consider Figure~\ref{fig:sections}. The rank invariant requires that all the elements of a square have class to contribute to the rank of the square. 
However, using our construction, a class will persist between two elements if and only if there is a sequence of maps in the diagram such that the classes map into each other (or from each other). 
In this case we can find persistent sections across incomparable elements yielding finer grained information than the rank invariant.
Furthermore, in highly structured diagrams such as multifiltrations, additional properties such as associativity have algorithmic consequences as well. 

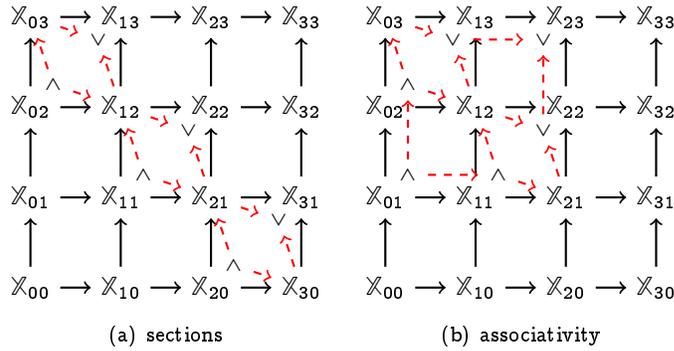
\begin{figure}
\subfigure[sections]{
     \begin{tikzpicture}[thick, scale=0.6] 
         \foreach \i  [evaluate=\i as \x using int(2*\i)] in {0,...,3}{
           \foreach \j [evaluate=\j as \y using int(2*\j)] in {0,...,3}{
             \node (p\i\j) at (\x,\y) {$\X_{\i\j}$}    ;
           }
         }
         \foreach \i [evaluate=\i as \x using int(\i+1)] in {0,...,2}{
           \foreach \j [evaluate=\j as \y using int(\j+1)]  in {0,...,2}{
             \draw[->] (p\i\j) -- (p\i\y)  ;
             \draw[->] (p\i\j) -- (p\x\j)  ;
           }
         }

         \foreach \j [evaluate=\j as \y using int(\j+1)]  in {0,...,2}{
           \draw[->] (p3\j) -- (p3\y)  ;
           \draw[->] (p\j3) -- (p\y3)  ;
           }
         \node (j1) at (2.5,2.5) {$\wedge$};
         \node (m1) at (3.5,3.5) {$\vee$};         
         \draw[->,dashed, red] (j1) -- (p12) ;
         \draw[->,dashed, red] (j1) -- (p21) ;
         \draw[->,dashed, red] (p12) -- (m1) ;
         \draw[->,dashed, red] (p21) -- (m1) ;
         \node (j) at (0.5,4.5) {$\wedge$};
         \node (m) at (1.5,5.5) {$\vee$};         
         \draw[->,dashed, red] (j) -- (p03) ;
         \draw[->,dashed, red] (j) -- (p12) ;
         \draw[->,dashed, red] (p12) -- (m) ;
         \draw[->,dashed, red] (p03) -- (m) ;
         \node (j2) at (4.5,0.5) {$\wedge$};
         \node (m2) at (5.5,1.5) {$\vee$};         
         \draw[->,dashed, red] (j2) -- (p30) ;
         \draw[->,dashed, red] (j2) -- (p21) ;
         \draw[->,dashed, red] (p30) -- (m2) ;
         \draw[->,dashed, red] (p21) -- (m2) ;
     \end{tikzpicture}
\label{exsections2}
}
\subfigure[associativity]{
     \begin{tikzpicture}[thick, scale=0.6]
         \foreach \i  [evaluate=\i as \x using int(2*\i)] in {0,...,3}{
           \foreach \j [evaluate=\j as \y using int(2*\j)] in {0,...,3}{
             \node (p\i\j) at (\x,\y) {$\X_{\i\j}$}    ;
           }
         }
         \foreach \i [evaluate=\i as \x using int(\i+1)] in {0,...,2}{
           \foreach \j [evaluate=\j as \y using int(\j+1)]  in {0,...,2}{
             \draw[->] (p\i\j) -- (p\i\y)  ;
             \draw[->] (p\i\j) -- (p\x\j)  ;
           }
         }
         \foreach \j [evaluate=\j as \y using int(\j+1)]  in {0,...,2}{
           \draw[->] (p3\j) -- (p3\y)  ;
           \draw[->] (p\j3) -- (p\y3)  ;
           }
         \node (j1) at (2.5,2.5) {$\wedge$};
         \node (m1) at (3.5,3.5) {$\vee$};         
         \draw[->,dashed, red] (j1) -- (p12) ;
         \draw[->,dashed, red] (j1) -- (p21) ;
         \draw[->,dashed, red] (p12) -- (m1) ;
         \draw[->,dashed, red] (p21) -- (m1) ;
         \node (j) at (0.5,4.5) {$\wedge$};
         \node (m) at (1.5,5.5) {$\vee$};         
         \draw[->,dashed, red] (j) -- (p03) ;
         \draw[->,dashed, red] (j) -- (p12) ;
         \draw[->,dashed, red] (p12) -- (m) ;
         \draw[->,dashed, red] (p03) -- (m) ;
         \node (j2) at (0.5,2.5) {$\wedge$};
         \node (m2) at (3.5,5.5) {$\vee$};         
         \draw[->,dashed, red] (j2) -- (j) ;
         \draw[->,dashed, red] (j2) -- (j1) ;
         \draw[->,dashed, red] (m1) -- (m2) ;
         \draw[->,dashed, red] (m) -- (m2) ;
     \end{tikzpicture}
\label{exsections1}
}
\caption{\label{fig:sections}
While the associativity of the lattice operations in the bifiltration corresponds to the possible paths in the diagram \subref{exsections1}, the sections in the lattice can be explained by the diagram \subref{exsections2}.
}
\label{exsections}
\end{figure}


\section{Discussion}
\label{discussion}

In this paper, we have investigated the properties of a lattice which contains information about the persistent homology classes in a general commutative diagram of vector spaces. There are still numerous open questions including:
\begin{itemize}
\item What kind of decompositions exist in the spirit of persistence diagrams for this distributive lattice, since all maximal total orders are the same length and therefore we can decompose this lattice into a canonical sequence of antitotal orders?
\item What are further algorithmic implications of this structure?
\item What is the correct metric to consider to general commutative diagrams as ``close''?
\item In what other contexts do such diagrams appear and what can we say about their structure?
\end{itemize}
We will address some of these questions in a subsequent paper.


\bibliographystyle{plain}

\begin{thebibliography}{10}

\bibitem{CatCS}
M. Barr and C. Wells.
\newblock {\em Category theory for computing science}, volume~10.
\newblock Prentice Hall, 1990.

\bibitem{Ba13}
A. Bauer, K. Cvetko-Vah, M. Gehrke, S. J. van Gool, and G. Kudryavtseva.
\newblock A non-commutative priestley duality.
\newblock {\em Topology and its Applications}, 2013.

\bibitem{Be10}
G. Bezhanishvili.
\newblock Bitopological duality for distributive lattices and Heyting algebras.
\newblock {\em Mathematical Structures in Computer Science}, 20(3):359--393,
  2010.

\bibitem{Be06}
N. Bezhanishvili.
\newblock {\em Lattices of intermediate and cylindric modal logics.}
\newblock Institute for Logic, Language and Computation, 2006.

\bibitem{Ba40}
G. Birkhoff.
\newblock {\em Lattice theory}, volume~5.
\newblock AMS Colloquium Publications, Providence RI, third edition, 1940.

\bibitem{Sa81}
S. Burris and H. P. Sankappanavar.
\newblock {\em A course in universal algebra}.
\newblock S. Burris and H. P. Sankappanavar, 1981.

\bibitem{TD}
G. Carlsson.
\newblock {Topology and data}.
\newblock {\em Bulletin-American Mathematical Society}, 46(2):1--54, 2009.

\bibitem{Car09}
G. Carlsson, V. De~Silva, and D. Morozov.
\newblock {Zigzag persistent homology and real-valued functions.}
\newblock In {\em Proceedings of the Annual Symposium on Computational
  Geometry}, pages 247--256, March 2009.

\bibitem{Carl09}
G. Carlsson and A. Zomorodian.
\newblock {The theory of multidimensional persistence}.
\newblock {\em Discrete {\&} Computational Geometry}, 42(1):71--93, 2009.

\bibitem{ZigZag}
G. Carlsson and V. de~Silva.
\newblock Zigzag persistence.
\newblock {\em Found. Comput. Math.}, 10(4):367--405, 2010.

\bibitem{Zom09}
G. Carlsson, G. Singh and A. Zomorodian.
\newblock {Computing Multidimensional Persistence}.
\newblock {\em arXiv.org}, July 2009.

\bibitem{Da02}
B. A. Davey and H. A. Priestley.
\newblock {\em Introduction to lattices and order.}
\newblock Cambridge University Press, 2002.

\bibitem{Dil50}
R. P. Dilworth.
\newblock A decomposition theorem for partially ordered sets.
\newblock {\em Annals of Mathematics}, 51(1):161--166, 1950.

\bibitem{Is08}
A. D\"{o}ring and C. Isham.
\newblock A topos foundation for theories of physics: I. formal languages for
  physics.
\newblock {\em Journal of Mathematical Physics}, 49, 2008.

\bibitem{Ed10}
H. Edelsbrunner and J. L. Harer.
\newblock {\em Computational Topology: An Introduction}.
\newblock American Mathematical Society, 2010.

\bibitem{Edel00}
H. Edelsbrunner, D. Letscher, and A. Zomorodian.
\newblock {Topological persistence and simplification}.
\newblock {\em Discrete {\&} Computational Geometry}, 28(4):511--533, December
  2012.

\bibitem{Fa96}
J. D. Farley.
\newblock The automorphism group of a function lattice: A problem of J\' onsson
  and McKenzie.
\newblock {\em Algebra Universalis}, 36(1):8--45, 1996.

\bibitem{Gr71}
G. Gr\"{a}tzer.
\newblock {\em Lattice theory}.
\newblock WH Freeman and Co, San Francisco, 1971.

\bibitem{Gr79}
G. Gr\"{a}tzer.
\newblock {\em Universal Algebra}.
\newblock Springer, second edition, 1979.

\bibitem{Hat00}
A. Hatcher.
\newblock {\em {Algebraic Topology}}.
\newblock Hatcher, December 2000.

\bibitem{Joh86}
P. T. Johnstone.
\newblock {\em {Stone Spaces}}.
\newblock Cambridge University Press, August 1986.

\bibitem{La87}
S. Lang.
\newblock {\em Linear Algebra}.
\newblock Springer, 3rd edition, 1987.

\bibitem{Lang}
S. Lang.
\newblock {\em Algebra}, volume 211.
\newblock Graduate texts in mathematics, 2002.

\bibitem{Lu00}
W. A. Luxemburg and A. C. Zaanen.
\newblock {\em Riesz spaces}.
\newblock North-Holland Publishing Company, 1971.

\bibitem{Le13}
J. Leech M. Kinyon and J. Pita Costa.
\newblock Distributive skew lattices.
\newblock {\em Submitted to Semigroup Forum}, 2013.

\bibitem{La98}
S. Mac~Lane.
\newblock {\em Categories for the Working Mathematician}.
\newblock Springer, 1998.

\bibitem{Mir71}
L. Mirsky.
\newblock A dual of Dilworth's decomposition theorem.
\newblock {\em American Mathematical Monthly}, 78(8):876--877, 1971.

\bibitem{Oud12}
S. Oudot and D. Sheehy.
\newblock {Zigzag Zoology: Rips Zigzags for Homology Inference}.
\newblock In {\em Proceedings of the twenty-ninth annual symposium on
  Computational geometry}, pages 387--396, 2012.

\bibitem{Be12}
H. Edelsbrunner P. Bendich, S. Cabello.
\newblock A point calculus for interlevel set homology.
\newblock {\em Pattern Recognition Letters}, 33(11):1436--1444, 2012.

\bibitem{Po05}
A. Polishchuk and L. Positselski.
\newblock {\em Quadratic Algebras, Clifford Algebras, and Arithmetic Witt
  Groups}.
\newblock American Mathematical Society, 2005.

\bibitem{Sk12}
P. {\v S}kraba and M. Vejdemo-Johansson.
\newblock Parallel and scalable zig-zag persistent homology.
\newblock In {\em NIPS 2012 Workshop on Algebraic Topology and Machine
  Learning.}, 2012.

\bibitem{Skr13}
P. {\v S}kraba and M. Vejdemo-Johansson.
\newblock {Persistence modules: Algebra and algorithms}.
\newblock {\em arXiv}, cs.CG, February 2013.

\bibitem{Tug98}
A. A. Tuganbaev.
\newblock {\em Semidistributive Modules and Rings}.
\newblock Mathematics and Its Applications. Springer, 1998.

\bibitem{Yo42}
K. Yosida.
\newblock On the representation of the vector lattice.
\newblock {\em Proceedings of the Japan Academy, Series A, Mathematical
  Sciences}, 18(7):339--342, 1942.

\bibitem{Zom13}
A. Zomorodian, G. Carlsson, A. Collins, and L. Guibas.
\newblock {Persistence Barcodes for Shapes}.
\newblock {\em International Journal of Shape Modeling (2oo5)}, pages 1--12,
  January 2013.

\bibitem{Zom05}
A. J. Zomorodian.
\newblock {\em {Topology for computing}}, volume~16.
\newblock Cambridge University Press, 2005.

\end{thebibliography}


%
\appendix
%


\section{Basics of Lattice Theory}
\label{Basics of Lattice Theory}

\subsection{Orders and lattice structures}
Partial orders are important tools in the study of topology.
Moreover, lattices are partially ordered sets (or posets for short) that have just enough structure to be seen as algebraic structures with operations determined by the underlying order structure.
In what follows we will provide the basic knowledge on the theory of lattices regarding the needs of this paper. 
For further reading on lattice theory and, in general, on universal algebra, we suggest \cite{Ba40}, \cite{Sa81}, \cite{Gr71} and \cite{Gr79}.

\begin{definition}
A \emph{preorder} is a binary relation $R$ that satisfies \emph{reflexivity} (i.e., for all $x\in A$, $xRx$) and \emph{transitivity} (i.e., for all $x,y,z\in A$, $xRy$ and $yRz$ implies $xRz$).
A preorder $\leq$ is a \emph{partial order} if , for all $x,y\in A$, $x\leq y$ and $y\leq x$ implies $x=y$ (antisymmetry). 
A \emph{poset} $(P, \leq)$ is an order structure consisting of a set $P$ and a partial order $\leq$. 
\end{definition}

\begin{example}
Examples of posets are the real numbers ordered by the standard order, the natural numbers ordered by divisibility, the set of subspaces of a vector space ordered by inclusion, or the vertex set of a directed acyclic graph ordered by reachability. 
\end{example}

\begin{definition}
A partial order is named \emph{total order} if every pair of elements is related, that is, for all $x,y\in A$, $x\leq y$ or $y\leq x$. 
On the other hand, an \emph{antitotal order} is a partial order for which no two distinct elements are related.
For every finite partial order there exists an antitotal order $S$ and a partition of the order in $A$ into a family $F$ of total orders such that the number of total orders in the partition equals the cardinality of $S$. Thus, $S$ must be the largest antitotal order in the order, and $F$ must be the smallest family of total orders into which the order can be partitioned (cf. \cite{Dil50}). 
Dually, The size of the largest total order in a partial order (if finite) equals the smallest number of antitotal orders into which the order may be partitioned (cf. \cite{Mir71}).
\end{definition}

\begin{example}
The natural numbers form a total order under the usual order, and form a partial order under divisibility. 
\end{example}

\begin{definition}
A \emph{lattice} is a poset for which all pairs of elements have an infimum and a supremum.
Whenever every subset of a lattice $L$ has a supremum and an infimum, $L$ is 
named a \emph{complete lattice}.
Every total order is a lattice. Other examples of lattices are the power set of A ordered by subset inclusion, or the collection of all partitions of A ordered by refinement.
A lattice $A$ can be seen as an algebraic structure $(L;\wedge ,\vee)$ with two 
operations $\wedge$ and $\vee$ satisfying \emph{associativity} (i.e., $x\wedge (y\wedge z) = (x \wedge y)\wedge z$ and $x\vee (y\vee z) = (x \vee y)\vee z$), \emph{idempontence} (i.e., $x\wedge x = x = x\vee x$), \emph{commutativity} (i.e., $x\wedge y=y\wedge x$ and $x\vee y=y\vee x$) and absorption (i.e., $x\wedge (x\vee y)=x=x\vee (x\wedge y)$).
The equivalence between this algebraic perspective of a lattice $L$ and its 
ordered perspective is given by the following: for all $x,y\in L$, $x\leq y$ iff 
$x\wedge y=x$ iff $x\vee y=y$.
\end{definition}


\begin{example}\label{partlat}
Recall that an \emph{equivalence} $E$ in a set $A$ is a preorder such that, for all $x,y\in A$, $xEy$ implies $yEx$ (symmetry). 
The set $A/E =\set{{x\in A: xEa} : a\in A}$ is a partition of $A$.
Conversely, every partition $P$ of $A$ determines an equivalence $\theta_{P}$ of $A$ defined 
by $x\theta_{P} y$ iff there exists $X\in P$ such that $x,y\in X$.
Thus the notions of equivalence relation and partition are essentially the same. 
The axiom of choice guarantees for any partition of a set $X$ the existence of a subset of $X$ containing exactly one element from each part of the partition. 
This implies that given an equivalence relation on a set one can select a canonical representative element from every equivalence class.
Arithmetical equality and geometrical similarity are examples of well known equivalences.
The partition of a set $X$ into nonempty and non-overlapping subsets, called \emph{blocks} (or \emph{cells}), determines a complete lattice for which the meet operation $\wedge$ is the intersection of blocks. 
\end{example}


\begin{example}\label{sub slat}
It is well known that the subspaces of a vector space form a complete lattice (cf. \cite{Ba40}). 
In fact, considering the partial order structure to be the subspace relation, whenever $A$ and $B$ are vector spaces one can define the lattice operations as $A\wedge B=A\cap B$ and $A\vee B=A\oplus B$. 
The minimum of this lattice is the trivial subspace $\set{0}$ while the maximum is the full vector space $V$. 
Clearly, $A\cap B=A$ iff $A\subseteq B$ iff $A\oplus B=B$.
Furthermore, $A\cap A\equiv A$, $A\oplus A\equiv A$, $A\cap (B\cap C)\equiv (A\cap B)\cap C$, $A\oplus (B\oplus C)\equiv (A\oplus B)\oplus C$ and
$A\cap (A\oplus B)\equiv A\cap A\oplus (A\cap B).$
If $V$ is a finite dimensional vector space over the field $K$ and $W\leq V$, then there exists $U\leq V$ such that $V=W\oplus U$ (\cite{La87}). On the other hand, $U\cap W=\set{0}$ giving us a sense of complement. 
This complement is not unique (and thus the lattice cannot be cancellative, or equivalently, the lattice is not distributive as will later be discussed). 
A linear lattice is a sublattice of the equivalences lattice of a set, on which any two elements commute.
A typical example can be found in Geometry: the lattice of subspaces of a vector space is isomorphic to a commuting equivalencesÕ lattice, deÞned in the vector space seen as a set.
If $V$ is a vector space and $W$ is one of its subspaces, we deÞne the equivalence of two vectors $x,y\in V$ as $x\equiv_{W} y$ if, and only if, $x-y\in W$, associating to each subspace an equivalence. If $W'$ is another subspace of $V$, then the equivalences $\equiv_{W}$ and $\equiv_{W'}$ commute, describing an isomorphism between the lattice $L(V)$ of all vector subspaces of $V$ and a lattice of commuting equivalences, $(Eq_{com}(V)  \cap L(V^{2});\cap,\circ)$.
Such lattices are of frequent occurrence, including the lattice of normal subgroups of a group, or the lattice of ideals of a ring. 
\end{example}


\begin{example}\label{vectlat}
A \emph{vector lattice} (or \emph{Riesz space}) $E$ is any vector space endowed with a partial order $\leq$ such that $(E;\leq)$ is a lattice and, for all vectors $x,y,z\in E$ and any scalar $\alpha \geq 0$: $x\leq y$ implies $x+y\leq y+z$, and $x\leq y$ implies $\alpha x\leq \alpha y$.
Given a topological space $X$ , its ring of continuous functions $C(X)$ is a vector lattice. 
In particular, any finite dimensional Euclidean space $\R^{n}$ is a vector lattice.
Roughly, vector lattice is a partially ordered real vector space where the order structure is a lattice.
A representation of such an algebraic structure is given in \cite{Yo42} assuming the Archimedean-unit and describing a representation space using maximal prime ideals. Being vector spaces, subalgebras are just subspaces that constitute sublattices. 
Riesz spaces have wide range of applications, having a great impact in measure theory. 
A large discussion on this topic can be found in \cite{Lu00}.
A Banach space is any complete normed vector lattice. Examples of such lattices are $C^{*}$ algebras, constituting associative algebras over the complex numbers which are Banach spaces with an involution map.
\end{example}

\begin{definition}
A lattice $L$ is \emph{complete} if every subset $S$ of $L$ has both a greatest lower bound $\bigwedge S$ and a least upper bound $\bigvee S$ in $L$.
In particular, when $S$ is the empty set, $\bigwedge S$ is the greatest element of $L$. Likewise, $\bigvee \emptyset$ yields the least element. 
Complete lattices constitute a special class of bounded lattices. Any lattice with arbitrary meets and a biggest element is complete. This condition and its dual characterize complete lattices.
\end{definition}

\begin{example}
Examples of complete lattices are abundant: the power set of a given set ordered by inclusion with arbitrary intersections and unions as meets and joins; the non-negative integers ordered by divisibility where the operations are given by the least common multiple and the greatest common divisor; the subgroups of a group, the submodules of a module or the ideals of a ring ordered by inclusion; the unit interval $[0,1]$ and the extended real number line, with the familiar total order and the ordinary suprema and infima. A totally ordered set with its order topology is compact as a topological space if it is complete as a lattice (cf. \cite{Gr71}).
\end{example}

\begin{remark}
There are several mathematical concepts that can be used to represent complete lattices being the Dedekind-MacNeille completion one of the most popular ones. It is used to extend a poset to a complete lattice. By applying it to a complete lattice one can see that every complete lattice is isomorphic to a complete lattice of sets. When noting that the image of any closure operator on a complete lattice is again a complete lattice one obtains another representation: since the identity function is a closure operator too, this shows that complete lattices are exactly the images of closure operators on complete lattices. 
\end{remark}


\subsection{Boolean algebras and Heyting algebras}
Maybe due to their important role as models to classical logic (constructive logic), Boolean algebras (Heyting algebras, respectively) are some of the the most well known lattices in Mathematics. 
We will present these varieties of algebras in the following paragraphs, discuss their important properties and present some examples.

 
\begin{definition}
A lattice $L$ is \emph{modular} if, for all $x,y,z\in S$, $y\leq x$ implies $x\wedge (y\vee z)=y\vee (x\wedge z)$. 
A lattice $L$ is \emph{distributive} if, for all $x,y,z\in S$, it satisfies one of the following equivalent equalities:

\begin{itemize}
\item[(d1)] $x\wedge (y\vee z)=(x\wedge y)\vee (x\wedge z)$; 
\item[(d2)] $x\vee (y\wedge z)=(x\vee y)\wedge (x\vee z)$;
\item[(d3)] $(x\vee y) \wedge (x\vee z) \wedge (y\vee z) = (x\wedge y)\vee (x\wedge z)\vee (y\wedge z)$. 
\end{itemize}

\end{definition}

\begin{example}
The lattice of normal subgroups of a group is modular (cf. \cite{Ba40}, V).
Other examples of modular lattices are the elements of any projective geometry or the ideals of any modular lattice (under set-inclusion)
A lattice of subsets of a set is usually called a \emph{ring of sets}. Any ring of sets forms a distributive lattice in which the intersection and union operations correspond to the lattice's meet and join operations, respectively. Conversely, every distributive lattice is isomorphic to a ring of sets; in the case of finite distributive lattices, this is \emph{Birkhoff's Representation Theorem} and the sets may be taken as the lower sets of a partially ordered set. Every field of sets and so also any $\sigma$-algebra also is a ring of sets (cf. \cite{Ba40}).
\end{example}

\begin{remark}
Below are the Hasse diagrams of the diamond $\mathbf{M}_{3}$ and the pentagon $\mathbf{N}_{5}$, the forbidden algebras regarding distributivity in lattices.

\begin{center}
$\begin{array}{ccc}

\begin{tikzpicture}[scale=.7]

  \node (1) at (0,1){$1$} ;
  \node (a) at (-1,0){$a$} ;
  \node (b) at (0,0){$b$};
  \node (c) at (1,0){$c$}  ;
  \node (0) at (0,-1){$0$} ;
  
    \draw (1) -- (c) -- (0) -- (a) -- (1) -- (b) -- (0);

\end{tikzpicture}

&

&

\begin{tikzpicture}[scale=.7]

  \node (1) at (0,1){$1$} ;
  \node (a) at (-1,0){$a$} ;
  \node (b) at (1,0.5) {$b$} ;
  \node (c) at (1,-0.5) {$c$} ; 
  \node (0) at (0,-1) {$0$} ;
  
      \draw (1) -- (a) -- (0) -- (c) -- (b) -- (1) ;

\end{tikzpicture}\\

\mathbf{M}_{3}

&
&

\mathbf{N}_{5}

\end{array}$
\end{center}

The following are useful characterizations of the distributivity and modularity of a lattice $L$:
\begin{itemize}
\item[(i)] $L$ is modular iff all $x,y,z$ with $y\leq z$ are such that $x\wedge y=x\wedge z$ and $x\vee y=x\vee z$ imply $y=z$;
\item[(ii)] $L$ is distributive iff all $x,y,z$ are such that $x\wedge y=x\wedge z$ and $x\vee y=x\vee z$ imply $y=z$;
\item[(iii)] $L$ is modular iff if it does not have embedded any copy of the pentagon $\mathbf{N}_{5}$;
\item[(iv)] $L$ is distributive iff if it does not have embedded any copy of the diamond $\mathbf{M}_{3}$ or of the pentagon $\mathbf{N}_{5}$.
\end{itemize}
\end{remark}

\begin{remark}
The modularity of distributive lattices also determines the \emph{diamond isomorphism theorem} describing the isomorphism between $[a \wedge b, b]$ and $[a, a \vee b]$ using the maps $f: (a\vee b)/a \rightarrow b/(a\wedge b)$, $x\mapsto x\wedge b$, and $g:  b/(a\wedge b)\rightarrow (a\vee b)/a$, $y\mapsto a\vee y$. This result is equivalent to the 3rd isomorphism theorem in Group Theory, being a particular case of the Correspondence Theorem established in the domain of Universal Algebra.
\end{remark}

\begin{example}
While every vector lattice, defined in Example \ref{vectlat}, is distributive (see \cite{Ba40}), the subspace lattices defined in Example \ref{sub slat}) are modular but not distributive: all indecomposable triples of vector spaces $X$, $Y$ and $Z$ but one are distributive; the only nondistributive indecomposable triple is that of three lines in a plane (cf. \cite{Po05}). 
Let $V$ be a 2-dimensional vector space. Consider the sublattice of the subspace lattice where the bottom element is the zero space, the top element is $V$, and the rest of the elements of $Sub(V)$ are 1-dimensional: lines through the origin. 
For 1-dimensional spaces, there is no relation $a\leq b$ unless $a$ and $b$ coincide. 
The Hasse diagram of such a lattice is the diamond $\mathbf{M}_{3}$ above where $1=V$, the total space.
Observe that for distinct elements $a,b,c$ in the middle level, we have for example $x\wedge y=0=x\wedge z$ ($0$ is the largest element contained in both $a$ and $b$), and also for example $b\vee c=1$ ($1$ is the smallest element containing $b$ and $c$). 
It follows that $a\wedge (b\vee c)=a\wedge 1=a$ whereas $(a\wedge b)\vee (a\wedge c)=0\vee 0=0$. 
The distributive law thus fails.
\end{example}

\begin{example}
Also a module $M$ over a ring $R$ can be considered a lattice with operations $+$ and $\cdot$ as $\vee$ and $\wedge$, respectively.
Lattice modularity corresponds to the Jordan-Dedekind total order condition.   
Moreover, $M$ is distributive iff for all $m,n \in M$, $(m+n)R=mI+nI$ for some ideal $I$ (cf. \cite{Tug98}).
The condition is easily seen to be necessary. For sufficiency, observe that distributivity is equivalent to $(m+n)R=(m+n)R\cap mR+(m+n)R\cap nR$ and to prove this, the argument says: the modular law implies that $(m+n)\cap RmR=(mI+nI)\cap mR=mI$ and respectively for $nI$.
\end{example}

\begin{example}
Moreover, the lattice of subgroups of a group ordered by inclusion is a modular lattice that is not distributive: consider $G$ to be the non-cyclic group of order 4, and $a$, $b$ and $c$ the three subgroups of order 2 having two distinct elements. We thus get the copy of $M_{3}$ in the Hasse diagram above (cf. \cite{Joh86}).
\end{example}

\begin{example}
Furthermore, the partition lattice, defined in Example \ref{partlat}, is not distributive for $n>3$ and is not modular for $n>4$. In detail just consider the following Hasse diagrams of the correspondent forbidden algebras:

\begin{center}
$\begin{array}{ccc}

\begin{tikzpicture}[scale=.7]

  \node (1) at (0,2){$123$} ;
  \node (a) at (-2,0){$1/23$} ;
  \node (b) at (0,0){$13/2$};
  \node (c) at (2,0){$12/3$}  ;
  \node (0) at (0,-2){$1/2/3$} ;
  
    \draw (1) -- (c) -- (0) -- (a) -- (1) -- (b) -- (0);

\end{tikzpicture}

&

&

\begin{tikzpicture}[scale=.55]

  \node (1) at (0,2.5){$1234$} ;
  \node (a) at (-2,0){$12/34$} ;
  \node (b) at (2,1) {$13/24$} ;
  \node (c) at (2,-1) {$1/3/24$} ; 
  \node (0) at (0,-2.5) {$1/2/3/4$} ;
  
      \draw (1) -- (a) -- (0) -- (c) -- (b) -- (1) ;

\end{tikzpicture}

\end{array}$
\end{center}

\end{example}


\begin{definition}
A \emph{Boolean algebra} is a distributive lattice with a unary operation $\neg$ 
and nullary operations $0$ and $1$ such that for all elements $x,y,z\in A$ the 
following axioms hold:

\begin{itemize}
\item[$L_6$.] $a\vee 0 = a$ and $a\wedge 1=a$; 
\item[$L_7$.] $a\vee \neg a = 1$ and  $a\wedge \neg a = 0$.
\end{itemize}

\end{definition}

\begin{example}
Examples of Boolean algebras are the power set of any set $X$ ordered by inclusion, or the divisors $D_{n}$ of a natural number $n$ bigger than $1$ that is not divided by the square of any prime number. 
\end{example}

\begin{remark}
The following result permits us to identify 
a Boolean algebra by observation of its Hasse diagram. 
Whenever $L$ is a bounded distributive lattice, the following are equivalent:

\begin{itemize}
\item[(i)] $L$ is a Boolean algebra;
\item[(ii)] for all $x\in L$ there exists $y\in L$ such that $x\wedge y=0$ and $x\vee y=1$;
\item[(iii)] for all $x,y,z\in L$ such that $x\leq y\leq z$ there exists $w\in L$ such that $y\wedge w=x$ 
and $y\vee w=z$.  
\end{itemize}  

Due to this it is easy to observe that total orders are not Boolean algebras. 
The distributive lattice represented by the Hasse diagram below is not a Boolean algebra: 
consider the total order $\set{3,x,4}$ and observe that there is no $y\in L$ such that $x\wedge y=4$ and $x\vee y=3$.

\begin{center}
\begin{tikzpicture}[scale=0.7]

  \node (1) at (0,2){$1$} ;
  \node (2) at (-1,1){$2$} ;
  \node (3) at (1,1) {$3$} ;
  \node (a) at (-2,0) {$a$} ; 
   \node (ad) at (2,0){$a'$} ;
  \node (x) at (0,0) {$x$} ;
  \node (4) at (-1,-1) {$4$} ; 
  \node (5) at (1,-1) {$5$} ;
  \node (0) at (0,-2) {$0$} ;
    
      \draw (0) -- (4) -- (a) -- (2) -- (1) -- (3) -- (x) -- (4) ;
      \draw (0) -- (5) -- (ad) -- (3) -- (x) -- (4) -- (0) -- (5) -- (x) -- (2)  ;

\end{tikzpicture}
\end{center}
\end{remark}

%
%
\begin{definition}
A bounded lattice $L$ is a \emph{Heyting algebra} if, for all $a,b\in  L$ there is a greatest element $x\in L$ such that $a\wedge x\leq b$. This element is the \emph{relative pseudo-complement} of $a$ with respect to $b$ denoted by $a\Rightarrow b$. 
\end{definition}

\begin{example}
Examples of Heyting algebras are the open sets of a topological space, as well as all the finite nonempty total orders (that are bounded and complete). 
Furthermore, every complete distributive lattice $L$ is a Heyting algebra with the implication operation given by $x \Rightarrow y = \bigvee \set{x\in L \mid x \wedge a \leq b}$.
\end{example}


%
\section{Algebraic constructions}
\label{Algebraic Constructions}

\subsection{On limits and colimits}
In the following paragraphs of this appendix we shall recall the categorical nature of products and coproducts of vector spaces.
We will also recall the definitions of equalizer and coequalizer, give some examples, and discussing their relation to pullbacks and pushouts.


\begin{remark}
As any other poset, the set of vector spaces ordered by $\leq$ constitutes a category, denoted by $\VV$, considering vector spaces as elements and linear maps as morphisms. It is a subcategory of $\mathbf{R-mod}$, the category of $R$-modules and $R$-module homomorphisms.
Recall that in the category of modules over some ring $R$, the product is the cartesian product with addition defined componentwise and distributive multiplication. 
Thus, the direct product of vector spaces $A$ and $B$, noted by $A\times B$, is a vector space when we define the sum and product by scalar componentwise. It is the biggest vector space that can be projected into $A$ and $B$, simultaneously.
Recall also that $A\cup B$ is a subspace iff $A=B$. The smallest element of $\VV$ containing $A\cup B$ is $A + B=\set{a+b\mid a\in A, b\in B}.$
The \emph{direct sum} of vector spaces $A$ and $B$, noted by $A\oplus B$, is the smallest 
vector space which contains the given vector spaces as subspaces with minimal constraints.
The direct product is the \emph{categorical product}, noted by $\sqcap$: whenever $A$ and $B$ are vector spaces, the natural projections $\pi_{A}: A\times B\rightarrow A$ and $\pi_{B}: A\times B\rightarrow B$ show that $A\times B\leq A,B$; on the other hand, whenever $D$ is a vector space such that $D\leq A,B$ with maps $f:D\rightarrow A$ and $g:D\rightarrow B$, the map $f\times g:D\rightarrow A\times B$ defined by $f\times g(x)=(f(x),g(x))$ is well defined and unique up to isomorphism, ensuring us with the universal property.  
As well, the direct sum is the \emph{categorical coproduct}, noted by $\sqcup$: whenever $A$ and $B$ are vector spaces, the inclusion maps $i_{A}:A\rightarrow A\oplus B$ and $i_{B}:B\rightarrow A\oplus B$ show us that $A,B\leq A\oplus B$; on the other hand, whenever $C$ is a vector space such that $A,B\leq C$ with maps $f:A\rightarrow C$ and $g:B\rightarrow C$, the map $f\oplus g:A\oplus B\rightarrow C$ defined by $f\oplus g(x)=f\oplus g(x_{A}+x_{B})=f(x_{A})+f(x_{B})$ is well defined and unique up to isomorphism, ensuring us with the universal property.   
A \emph{biproduct} of a finite collection of objects in a category with zero object is both a product and a coproduct. 
In a preaddictive category the notions of product and coproduct coincide for finite collections of objects.
The biproduct generalizes the direct sum of modules. The category of modules over a ring is preaddictive (and also additive).
In particular, the category of vector spaces over a field is preaddictive with the trivial vector space as zero object.
\end{remark}


In the following we are going to discuss in detail the (generalized) categorical concepts of equalizer and coequalizer that we use in this paper to construct the lattice operations. 
We will also interpret these in the framework of persistence in order to use such ideas to construct the lattice operations. 

\begin{definition}
Given a pair of vector spaces $A$ and $C$ with two linear maps $f,g:A\Rightarrow B$ between them, the \emph{equalizer} of $f$ and $g$ is a pair $(E, e)$ where $E$ is a vector space (usually called \emph{kernel set} of the equalizer) and $e:E\rightarrow A$ is a linear map such that $fe=ge$, with the following universal property: for any other vector space $E'$ and linear map $e' : E' \rightarrow A$ such that $fe' = ge'$, there exists a unique linear map $\phi : E' \rightarrow E$ such that $e\phi = e'$ (as represented in the diagram of Figure \ref{figcoeq}). 

Dually, the $\emph{coequalizer}$ of $f$ and $g$ is a pair $(H, h)$ where $H$ is a vector space (usually called the \emph{quotient set} of the coequalizer) and $h:A\rightarrow H$ is a linear map such that $hf=hg$, with the following universal property: for any other vector space $H'$ and linear map $h' : A \rightarrow H'$ in $\VV$ such that $h'f = h'g$, there exists a unique morphism $\phi : H \rightarrow H'$ such that $\phi h = h'$ (as represented in the diagram of Figure \ref{figcoeq}).

\end{definition}

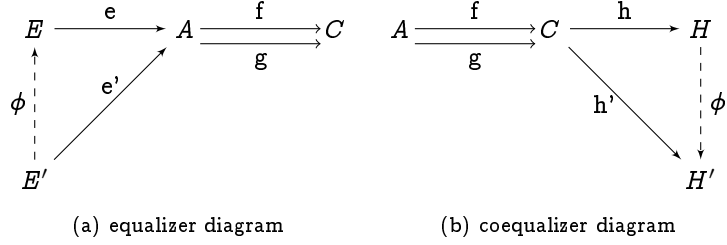
\begin{figure}
\subfigure[equalizer diagram]{
\begin{tikzpicture}

  \node (e) at (-2,0) {$E$};
  \node (a) at (0,0) {$A$};
  \node (b) at (2,0) {$C$};
  \node (f) at (-2,-2) {$E'$};
  
\draw[arrows=-latex'] (e) -- (a) node[pos=.5,above] {e};
\draw[arrows=-latex'](f) -- (a) node[pos=.5,above] {e'};
\draw[arrows=-latex', dashed] (f) -- (e) node[pos=.5,left] {$\phi$};
   
\draw[postaction={transform canvas={yshift=-2mm},draw}]
[->] (0.2,0) -- (1.8,0) node[pos=.5,above] {f} node[pos=0.5,below,yshift=-2mm] {g};

\end{tikzpicture}   
}
\subfigure[coequalizer diagram]{
\begin{tikzpicture}

  \node (e) at (2,0) {$H$};
  \node (a) at (-2,0) {$A$};
  \node (b) at (0,0) {$C$};
  \node (f) at (2,-2) {$H'$};
  
\draw[arrows=-latex'] (b) -- (e) node[pos=.5,above] {h};
\draw[arrows=-latex'](b) -- (f) node[pos=.5,left] {h'};
\draw[arrows=-latex', dashed] (e) -- (f) node[pos=.5,right] {$\phi$};
   
\draw[postaction={transform canvas={yshift=-2mm},draw}]
[->] (-1.8,0) -- (-0.2,0) node[pos=.5,above] {f} node[pos=0.5,below, yshift=-2mm] {g};

\end{tikzpicture}
}
\caption{Diagram representation of the equalizer and coequalizer of maps $f$ and $g$ between vector spaces $A$ and $C$ on a given diagram $\DD$.}
\label{figcoeq}
\end{figure}

\begin{example} 
%
%
In the category of sets, given maps $f,g: X\rightarrow Y$, the equalizer of $f$ and $g$ is the set $\set{x\in X\mid f(x)=g(x)}$ while the coequalizer of $f$ and $g$ is the quotient of $Y$ by the equivalence generated by the set $\set{(f(x),g(x)) | x \in X}$, i.e., the smallest equivalence $\theta$ such that for every $x\in X$, $f(x)\theta g(x)$ holds.
For instance, consider the sets $X = \set{a, c, d}$ and $Y = \set{a, c, d, e}$, and the maps $f:X\rightarrow Y = \set{a\mapsto e, c\mapsto d, d\mapsto c}$ and $g:X\rightarrow Y = \set{a\mapsto d, c\mapsto d, d\mapsto c}$. 
The equalizer of $f$ and $g$ is given by the kernel set $\EE = \set{c, d}$ and the injection $eq:E\rightarrow X = \set{c\mapsto c, d\mapsto d}$.
On the other hand, the coequalizer of $f$ and $g$ is given by the quotient set $\CC=\set{\set{a}, \set{c}, \set{d, e}}$ and the surjection $coeq:Y\rightarrow \CC = \set{a\mapsto \set{a}, c\mapstoÊ\set{c}, d\mapsto \set{d, e}, e\mapsto \set{d, e}}$.

The equalizer of the real functions $f(x,y)=x^2+y^2$ and $g(x,y)=4$ is the 
circumference $E=\set{(x,y)\in \R^2\mid x^2+y^2=4}$ together with projection maps. 
\end{example}

\begin{example}
%
%
In the category of groups, the equalizer of homomorphisms $f,g: X\rightarrow Y$ can still be seen as the solution set of an equation determined by $f(x)=g(x)$ while their coequalizer is the quotient of $Y$ by the normal closure of the set $S=\set{f(x)g(x)^{-1}\mid x\in X}$. In detail, the elements of $Y/N$ must be equivalence classes $y/N$ such that, for all $y,y'\in Y$, 
\[
y\theta y' \text{  iff  } f.g^{-1}\in N.
\] 
In particular, for abelian groups, the equalizer is the kernel of the morphism $f-g$ while the coequalizer is the factor group $Y/im(f-g)$, i.e., the cokernel of $f-g$. Moreover, the kernel of a linear map $f$ is the equalizer of the maps $f$ and $0$ constituting a normal subgroup with the following property: for any normal subgroup $N\subseteq G$, $N\subseteq \ker f$ iff there is a (necessarily unique) homomorphism $h:X/N \rightarrow Y$ such that $h\circ \pi_{N}=f$ implying the commutativity of the diagram below: 

\begin{center}
\begin{tikzpicture}

  \node (e) at (-2,0) {$X$};
  \node (a) at (0,0) {$Y$};
  \node (f) at (-2,-2) {$X/N$};
  
\draw[arrows=-latex'] (e) -- (f) node[pos=.5,left] {$\pi_{N}$};
\draw[arrows=-latex'](e) -- (a) node[pos=.5,above] {$f$};
\draw[arrows=-latex', dashed] (f) -- (a) node[pos=.5,right] {$h$};
   
\end{tikzpicture}   
\end{center}

\noindent Hence, every group homomorphism factors as a quotient followed by an injective homomorphism (every group homomorphism has a kernel).
On the other hand, a coequalizer of a homomorphism $f: X \rightarrow Y$ and the zero homomorphism is the natural surjection $\pi_{f(X)} : Y \rightarrow Y/f(X)$ on the quotient $Y/f(X)$. More generally, a coequalizer of homomorphisms $f, g : X\rightarrow Y$ is a coequalizer of 
$f - g : X \rightarrow Y$ and the zero homomorphism, that is, the natural surjection $Y \rightarrow Y/(f - g)(X)$. 
This holds for the category of vector spaces and linear maps.
\end{example}

%
%

%
\begin{remark}
Now we will show that, for the purposes of this paper, the information retrieved by pullbacks and pushouts is essential the same than the one obtained by computing equalizers and coequalizers, respectively. 
It is well known that (pushouts) pullbacks can be constructed from (co)equalizers: a pullback is the equalizer of the morphisms $f \circ \pi_1$, $g \circ \pi_2 : X \times Y \rightarrow Z$ where $X \times Y$ is the binary product of $X$ and $Y$, and $\pi_1$ and $\pi_2$ are the natural projections, showing that pullbacks exist in any category with binary products and equalizers. In general, we have the following:
The equalizer of the family $(f_{i})_{i\in I}:A\oplus B \rightarrow C$ is the pullback of the pair of morphisms $((f_{iA})_{i\in I}, (f_{iB})_{i\in I})$ with $f_{iA}:A\rightarrow C$, $f_{iB}:B\rightarrow C$ and $f_{i}(z)=f_{iA}(x)+f_{iB}(y) \text{,  for all  } z=x+y\in A\oplus B \text{ and all } i\in I.$
Dually, the coequalizer of the family $(g_{i})_{i\in I}:D \rightarrow A\oplus B$ is exactly the pushout of the pair of morphisms $((g_{jA})_{j\in J}, (g_{jB})_{j\in J})$ with $g_{jA}:D\rightarrow A$, $g_{jB}:D\rightarrow B$ and $g_{j}(x)=g_{jA}(x)\oplus g_{jB}(x) \text{,  for all  } x\in D \text{ and all } j\in I.$.
To see this in detail, just observe that all maps $f_{i}:A\oplus B \rightarrow C$ split into maps $f_{iA}:A \rightarrow C$ and $f_{iB}:B \rightarrow C$ with $f_{j}(x)=f_{jA}(x)+f_{jB}(x) \text{,  for all  } j\in I.$ Hence, the  diagrams of Figure \ref{pbfig} are equivalent.
Clearly, $e=e_{A}+e_{B}$ defined by $e(x)=e_{A}(x)+e_{B}(x)$ and thus 
\[
\set{x\in A\oplus B\mid f_{i}(x)=f_{j}(x)}=\set{(x,y)\in A\times B\mid f_{i}(x)=f_{j}(y)}.
\]
The dual result has a similar argument.
\end{remark}
\begin{figure}
\subfigure[equalizer diagram]{
\begin{tikzpicture}
  \node (e) at (-2,0) {$E$};
  \node (a) at (0,0) {$A\oplus B$};
  \node (b) at (2,0) {$C$};
%
\draw[arrows=-latex'] (e) -- (a) node[pos=.5,above] {e};
%
\draw[postaction={transform canvas={yshift=-2mm},draw}]
[->] (0.6,0) -- (1.6,0) node[pos=.5,above] {$f_{i}$} node[pos=0.5,below, yshift=-2mm] {$f_{j}$};
\end{tikzpicture}   
}
\subfigure[pullback diagram]{
\begin{tikzpicture}
\path (-1.5,1.5) node[] (S) {$E$};
\path (-1.5,-1.5) node[] (L) {$B$};
\path (1.5,1.5) node[] (R) {$A$};
\path (1.5,-1.5) node[] (D) {$C$};
\draw[arrows=-latex'] (S) -- (R) node[pos=.5,above] {$e_{A}$};
\draw[arrows=-latex'](S) -- (L) node[pos=.5,left] {$e_{B}$};
\draw[postaction={transform canvas={yshift=-2mm},draw}]
[->] (-1,-1.5) -- (1,-1.5) node[pos=.5,above] {$f_{i}$} node[pos=0.5,below, yshift=-2mm]{$f_{j}$};
\draw[postaction={transform canvas={xshift=-2mm},draw}]
[->] (1.5,1) -- (1.5,-1) node[pos=.5,right] {$f_{i}$} node[pos=0.5,left, xshift=-2mm] {$f_{j}$};
\end{tikzpicture}
}
\caption{Equivalence of the considered equalizer diagrams and pullback diagrams.}
\label{pbfig}
\end{figure}
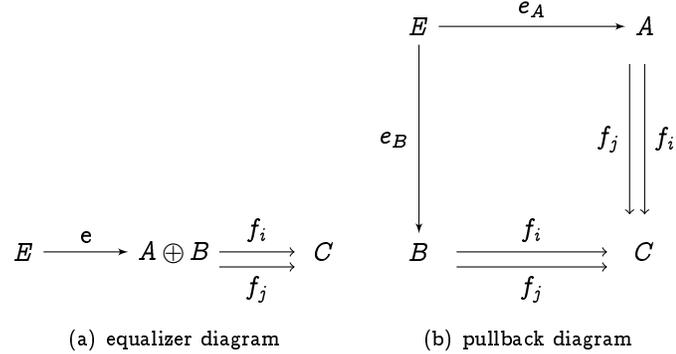

\begin{remark}
Equalizers (coequalizers) are unique up to isomorphism. The equalizing map $e$ (coequalizing map $h$) is always a monomorphism (epimorphism) and, monomorphisms (epimorphisms) are injective (surjective) maps in the context of vector spaces and linear maps. 
Hence, the equalizing map $e$ (coequalizing map $h$) is an isomorphism iff $f=g$, for all $f,g\in Hom(A,C)$ (cf. \cite{La98}).
\end{remark}

\begin{remark}
Sometimes, the equalizer is identified with the object $E$ while the morphism $e$ can be taken to be the inclusion map of $E$ as a subset of $A$. Dual remarks hold for coequalizers.
As we assume that all diagrams of vector spaces commute, the categorical concepts of \emph{equalizer} and \emph{coequalizer} can be adapted to the framework of this paper as in Definition \label{coeq}.
\end{remark}


\subsection{A construction for the lattice operations}
In the following we will discuss a natural generalization of the equalizer and coequalizer constructions where the definition of the lattice operations exhibited in Section \ref{Structural Theory} is based. 

\begin{remark}
Let us first recall that linear maps from common sources are maps from the direct sum of those sources, and that linear maps to common targets are maps to the direct sum of those targets.  
To see this consider the vector spaces $A$, $B$, $C$ and $D$, and the linear maps $f:A\rightarrow C$ and $g:B\rightarrow C$. 
Then we can construct $f\oplus g:A\oplus B \rightarrow C$ defining it by $f\oplus g(z)=f(x) + g(y)$ for all $z=x+y\in A\oplus B$.
Moreover, given the linear maps $f:A\rightarrow C$ and $h:A\rightarrow B$ we can construct $f\oplus h:A\rightarrow C\oplus B$ by defining it as $f\oplus h(x)=f(x)\oplus h(x)$, for all $x\in A$.
Conversely, any linear map $f:A\oplus B\rightarrow C$ ``splits'' to maps $f_{A}:A\rightarrow C$ and $f_{B}:B\rightarrow C$ such that $f(x)=f_{A}(x)+f_{B}(x)$ for all $x\in A\oplus B$. Dually, a map $g: D\rightarrow A\oplus B$ can also ``split'' into maps $g_{A}:D\rightarrow A$ and $g_{B}:D\rightarrow B$ such that $g(x)=g_{A}(x)+g_{B}(x)$ for all $x\in A\oplus B$ with $g_{A}(x)\in A$ and $g_{B}(x)\in B$. 
\end{remark}

\begin{remark}\label{existb}
Whenever $A$, $B$, $C$ and $D$ are vector spaces of a given diagram $\DD$,
\begin{itemize}
\item[(i)] if $A,B\leq C$ then $A\oplus B\leq C$;
\item[(ii)] if $D\leq A,B$ then $D\leq A\oplus B$.
\end{itemize}
Indeed, $(i)$ follows from the fact that the direct sum is the coproduct in the category of vector spaces and linear maps. 
To see $(ii)$ consider the inclusion maps $i_{A}:A\rightarrow A\oplus B$ and $i_{B}:B\rightarrow A\oplus B$ and observe that, due to the hypothesis, there exist maps $f:D\rightarrow A$ and $g:D\rightarrow B$. Thus, the compositions $i_{A}\circ f$ and $i_{B}\circ g$ ensure the inequality $D\leq A\oplus B$.  
\end{remark}


\begin{definition}
Let $A$ and $B$ be vector spaces and $I$ and $J$ be arbitrary sets. 
Consider the family of linear maps from $A\oplus B$ to all vector spaces with common sources $A$ and $B$, i.e., 
\[
\Ff_k = \{f_i: A\oplus B\rightarrow X_k \mid \text{ for all vector spaces } X_k\geq A,B \text{ and }  i\in I \}
\] 
and, dually, the family of linear maps from all vector spaces with common targets $A$ and $B$ to $A\oplus B$, i.e.,
\[
\G_k = \{g_i: Y_k\rightarrow A\oplus B \mid \text{ for all vector spaces } Y_k \leq A,B \text{ and } i\in I \}.
\]

Define $A\wedge B$ to be the kernel set $\EE $ of the equalizer of the linear maps of the family $\Ff_{k}$, $eq(\oplus_{k\in J} \Ff_{k})$, and $A\vee B$ to be the quotient set $\CC$ of the coequalizer of the linear maps of the family $\G_{k}$, $coeq(\oplus_{k\in J} \G_{k})$. 
These operations are well defined due to Remark \ref{existb}.
Moreover, as all considered maps on the construction of the kernel set $A\wedge B$ and the quotient set $A\vee B$ are linear, $A\wedge B$ is a subalgebra of $A\oplus B$ and $A\vee B$ is a quotient algebra of $A\oplus B$. Both of them constitute vector spaces.
\end{definition}


\begin{remark}
We shall discuss now the equalizer set and quotient set constituting the meet and the join, respectively, of vector spaces in a given diagram.
Whenever $A$ and $B$ are vector spaces with common targets $C_{1}$ and $C_{2}$, i.e., such that $A,B\leq C_{1}$ and $A,B\leq C_{2}$, we can consider linear maps $f_{1},g_{1}:A \oplus B \rightarrow C_{1}$ and $f_{2},g_{2}:A \oplus B \rightarrow C_{2}$, and the equalizers $eq(f_{1},g_{1})$ and $eq(f_{2},g_{2})$ with kernel sets $\EE_{1}=\set{x\in A \mid f_{1}(x)=g_{1}(x)}$ and $\EE_{2}=\set{x\in A \mid f_{2}(x)=g_{2}(x)}$, respectively. 
Define $eq(f_{1},g_{1}) \vee eq(f_{2},g_{2})$ to be the pair $(E,e)$ with kernel set determined by the union of equations in $\EE_{1}$ and in $\EE_{2}$, i.e., 
\[
\EE_{1,2}=\set{x\in A\oplus B \mid f_{k}(x)=g_{k}(x), k\in \set{1,2}}
\]
and corresponding inclusion map $e:E\hookrightarrow A\oplus B$. 
This new pair is an equalizer of all the considered maps from $A\oplus B$ to $\oplus_{k\in \set{1,2}} C_{k}$ (as represented in Figure \ref{figlim} (a)). 
Indeed $\EE_{1,2}=\EE_{1}\cap \EE_{2}$: if we look at $\EE_{k}$ as a set of equations, $\EE_{1,2}$ is determined by both the defining equations in $\EE_{1}$ and $\EE_{2}$, i.e., 
\[
\EE_{1}\cap \EE_{2}=\set{x\in A\oplus B\mid (f_1(x),f_2(x))=(g_1(x),g_2(x))}=\set{x\in A\oplus B\mid f_1\oplus f_2 (x) = g_1\oplus g_2 (x)}.
\] 

Dually, whenever $A$ and $B$ are vector spaces with common sources $D_{1}$ and $D_{2}$, i.e., such that $D_{1}\leq A,B$ and $D_{2}\leq A,B$ we can consider linear maps $f_{1},g_{1}:D_{1} \rightarrow  A\oplus B$ and $f_{2},g_{2}:D_{2} \rightarrow A \oplus B  $. 
The quotient sets of the coequalizers $coeq(f_{1},g_{1})$ and $coeq(f_{2},g_{2})$ are quotients of $A\oplus B$ by the equivalences $\theta_{1}= \langle \set{(f_{1}(x),g_{1}(x))\mid x\in D_{1}}\rangle$ and $\theta_{2}=\langle \set{(f_{2}(x),g_{2}(x))\mid x\in D_{2}}\rangle$, respectively. 
Define $coeq(f_{1},g_{1}) \vee coeq(f_{2},g_{2})$ to be the pair $(H,h)$ with underlying set constituted by the quotient of $A\oplus B$ by the equivalence $\theta $ generated by 
\[
\langle \theta_{1}\cup \theta_{2}\rangle=\langle \set{(f_{k}(x),g_{k}(x))\mid  x\in D_{1}\cap D_{2}, i\in \set{1,2}}\rangle
\] 
and corresponding linear map $h:A\oplus B\hookrightarrow A\oplus B/\theta$. 
This new pair is a coequalizer of all the considered maps from $\oplus_{k\in \set{1,2}} D_{k}$ to $A\oplus B$ (as represented in Figure \ref{figlim} (b)). 
Whenever $D_{1}\cap D_{2}=\set{0}$, the respective equivalence $\theta $ is $\langle (0,0) \rangle = \set{(0,0)}=0$ and thus $\CC_{1,2}=A\oplus B / 0 \cong A\oplus B$.
Observe that we are generating the equivalence that includes all the possible pairs given by the linear maps to each $D_{k}$. 
In fact, the union of equivalences is not, in general, an equivalence but it is clearly included in the equivalence generated by this union. 
The quotient by this bigger equivalence $\theta$, generated by the union of all the others, will correspondent to the smallest quotient above $A$ and $B$ in the requested conditions. 
To see this consider the partition semilattice of quotients of $A$ with the meet operation defined as 
\[
A/\theta_{1} \wedge A/\theta_{2}=\set{x/\theta_{1}\cap x/\theta_{2}\mid x\in A}=A/(\theta_{1}\cap \theta_{2}).
\] 
Thus, $A/\theta \subseteq  A/\theta_{1} \wedge A/\theta_{2}$. 
In general, whenever $\theta$ is the equivalence generated by the union of the equivalences $\theta_{k}$ corresponding to each vector space $Y_{k}$ above $A$ and $B$, then 
\[
A/\theta \subseteq \bigwedge_{k\in J} (A/\theta_{k}).
\]
\end{remark}

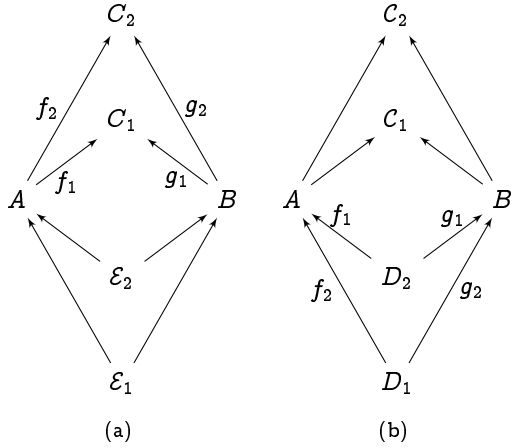
\begin{figure}
\subfigure[]{
\begin{tikzpicture}[scale=.7]

  \node (1) at (0,3.5) {$C_{2}$};
  \node (c) at (0,1.5) {$C_{1}$};
  \node (a) at (-2,0) {$A$};
  \node (b) at (2,0) {$B$};
  \node (d) at (0,-1.5) {$\EE_{2}$};
  \node (0) at (0,-3.5) {$\EE_{1}$};
  
\draw[arrows=-latex'] (a) -- (1) node[pos=.5,left] {$f_{2}$};
\draw[arrows=-latex'](b) -- (1) node[pos=.5,right] {$g_{2}$};
\draw[arrows=-latex'] (0) -- (a) node[pos=.5,above] {};
\draw[arrows=-latex'](0) -- (b) node[pos=.5,right] {};
\draw[arrows=-latex'] (a) -- (c) node[pos=.5,below] {$f_{1}$};
\draw[arrows=-latex'](b) -- (c) node[pos=.5,below] {$g_{1}$};
\draw[arrows=-latex'] (d) -- (a) node[pos=.5,above] {};
\draw[arrows=-latex'](d) -- (b) node[pos=.5,right] {};

\end{tikzpicture}  
}
\subfigure[]{
\begin{tikzpicture}[scale=.7]

  \node (1) at (0,3.5) {$\CC_{2}$};
  \node (c) at (0,1.5) {$\CC_{1}$};
  \node (a) at (-2,0) {$A$};
  \node (b) at (2,0) {$B$};
  \node (d) at (0,-1.5) {$D_{2}$};
  \node (0) at (0,-3.5) {$D_{1}$};
  
\draw[arrows=-latex'] (a) -- (1) node[pos=.5,left] {};
\draw[arrows=-latex'](b) -- (1) node[pos=.5,right] {};
\draw[arrows=-latex'] (0) -- (a) node[pos=.5,left] {$f_{2}$};
\draw[arrows=-latex'](0) -- (b) node[pos=.5,right] {$g_{2}$};
\draw[arrows=-latex'] (a) -- (c) node[pos=.5,below] {};
\draw[arrows=-latex'](b) -- (c) node[pos=.5,below] {};
\draw[arrows=-latex'] (d) -- (a) node[pos=.5,above] {$f_{1}$};
\draw[arrows=-latex'](d) -- (b) node[pos=.5,above] {$g_{1}$};

\end{tikzpicture}
}
\caption{Diagram representation of the meet and join of vector spaces $A$ and $B$ of a given diagram $\DD$ when $A$ and $B$ have more than one common target (case (a) with targets $C_1$ and $C_2$) or more than one common source (case (b) with sources $D_1$ and $D_2$), respectively.}
\label{figlim}
\end{figure}

\begin{proposition}\label{general}
Let $I$ be an index set and $A$, $B$, $C_{i}$ and $D_{j}$ be vector spaces such that $D_{j}\leq A,B \leq C_{i}$, for all $i,j\in I$.
Consider the families of linear maps $\FF_{k}=\set{f_{ik}:A \oplus B \rightarrow C_{k}}$, $\FF'_{k}=\set{f'_{ik}:D_{k}\rightarrow A\oplus B }$, for some $k\leq i,j$.
Consider also the equalizers $eq(\FF_{k})=(E_{k},e_{k})$ and the coequalizers $coeq(\FF'_{k})=(H_{k},h_{k})$.
Then, 
\begin{itemize}
\item[(i)] the kernel set $E_k=\EE ((F_{k})_{k\in I})$ is the intersection of all the kernel sets corresponding to the equalizers of linear maps of the family $(F_{k})_{k\in I}$.  
\item[(ii)] the quotient set $H_k=\CC((F'_{k})_{k\in I})$ is constituted by the quotient of $A\oplus B$ by the equivalence generated by the union of all equivalences corresponding to the family of linear maps from $(F'_{k})_{k\in I}$. 
\end{itemize}
\end{proposition}

\begin{proof}
Consider the kernel set of  the equalizer $eq((F_{k})_{k\in I})$ given by
 \[
 \EE=\bigcap_{k\in J} \set{eq(f_{ik},f_{jk}) \mid f_{ik},f_{jk}\in Hom(A\oplus B, D_k)} \text{,  that is,} 
 \]
 
 \[
 \EE= \set{x\in A \oplus B \mid f_{ik}(x)=f_{jk}(x)\text{, for some } f_{ik},f_{jk}\in \bigcup_{k\in J} Hom(A\oplus B, D_k)}.
 \]
 
The corresponding linear map $e$ is the inclusion map $E\hookrightarrow A\oplus B$.
Furthermore, the universal property derives from the conjugation of the universal properties valid to each equalizer $eq(\FF_{k})$.

Dually, observe that, for each $k\in J$, the quotient set of the coequalizer $coeq((F_{k})_{k\in I})$ is given by: the factor $A \oplus B/\theta$, where $\theta $ is the equivalence generated by the set 
 \[
 \bigcap_{k\in J} \set{(f_{ik}(x),f_{jk}(x)) | x \in Y_k \text{  and  } f_{ik},f_{jk}\in Hom(D_{k}, A\oplus B)} \text{,  that is,} 
 \]
 
 \[
  \set{(f_{ik}(x),f_{jk}(x)) | x \in \bigcup_{k\in J} Y_k \text{  and  } f_{ik},f_{jk}\in \bigcup_{k\in J} Hom(D_{k}, A\oplus B)}.
 \]
 
 The corresponding linear map $h$ is the canonical projection map $A\oplus B\hookrightarrow A\oplus B/\theta$.
Furthermore, the universal property again derives from the conjugation of the universal properties valid to each coequalizer $coeq(\FF_{k})$.
\end{proof}

\begin{corollary}\label{generalb}
Let $C_1<C_2<\dots C_n$, $D_m<D_{m-1}<\dots C_1$ and $A,B$ be vector spaces in a diagram such that $D_1<A,B<C_n$. 
Then, the equalizer of $\bigcup_k Hom(A\oplus B, C_k)$ is just the equalizer of $Hom(A\oplus B, C_n)$ 
while  the coequalizer of $\bigcup_k Hom(D_k, A\oplus B)$ is just the coequalizer of $Hom(D_m,A\oplus B)$.
\end{corollary}

\begin{proof}
This result is due to the assumption of the commutativity of all diagrams together with proposition \ref{general}.
\end{proof}

\begin{remark}
Both Proposition \ref{general} and Corollary \ref{generalb} now link to Theorem \ref{complete} establishing the completeness of persistence lattices.
Indeed, both of the lattice operations extend to arbitrary
joins $\bigjoin_i D_i$ given by
\[
\bigwedge S = \set{x\in X: f_{i}(x)=f_{j}(x) \text{,  for all  } f_{i},f_{j}\in \bigcup_{k}Hom(X,C_{k})}.
\]
and meets $\bigmeet_i D_i$ given by 
\[
\bigvee_{\ell} A_{\ell}=(\oplus_{\ell} A_{\ell})/\bigcap_k \langle (f_{i}(x),f_{j}(x))\mid x\in \oplus_{k} D_{k}\rangle
\]
that are a great deal dependent from the biggest element of the correspondent total orders determined by $\bigcup_{k}Hom(X,C_{k})$ and $\oplus_{k} D_{k}$, respectively.
\end{remark}


%
\section{Glossary of Definitions}
\label{Glossary of Definitions}

In the following we present a list of basic concepts of lattice theory and category theory that will help the reader, that is unfamiliar with such, through this paper. These concepts are presented by order of appearance. For more details please read \cite{Ba40}, \cite{Gr71} or \cite{La98}.

\begin{itemize}[label=\ding{75}]

\item Preorder $\equiv$ a binary relation $R$ that satisfies \emph{reflexivity} (i.e., for all $x\in A$, $xRx$) and \emph{transitivity} (i.e., for all $x,y,z\in A$, $xRy$ and $yRz$ implies $xRz$).

\item Partial order $\equiv$ a preorder $\leq$ such that, for all $x,y\in A$, $x\leq y$ and $y\leq x$ implies $x=y$ (antisymmetry).
 
\item Poset $\equiv$ an order structure $(P, \leq)$ consisting of a set $P$ and a partial order $\leq$. 

\item Total order $\equiv$ a poset such that every pair of elements is related, that is, for all $x,y\in A$, $x\leq y$ or $y\leq x$. 

\item Antitotal order $\equiv $ a partial order for which no two distinct elements are related.

\item Lattice $\equiv $ a poset for which all pairs of elements have an infimum and a supremum.

\item Complete lattice $\equiv $ a poset for which every subset has a supremum and an infimum.
 
\item Associativity $\equiv $ for all $x,y,z$, $x\wedge (y\wedge z) = (x \wedge y)\wedge z$ and $x\vee (y\vee z) = (x \vee y)\vee z$.

\item Idempotency $\equiv $ for all $x$, $x\wedge x = x = x\vee x$.

\item Comutativity $\equiv $ for all $x,y$, $x\wedge y=y\wedge x$ and $x\vee y=y\vee x$.

\item Absorption $\equiv $ for all $x,y$, $x\wedge (x\vee y)=x=x\vee (x\wedge y)$.

\item Modularity $\equiv $ for all $x,y,z$, $y\leq x$ implies $x\wedge (y\vee z)=y\vee (x\wedge z)$.

\item Distributivity $\equiv $ for all $x,y,z$,  $x\wedge (y\vee z)=(x\wedge y)\vee (x\wedge z)$ or $x\vee (y\wedge z)=(x\vee y)\wedge (x\vee z)$.

\item Heyting algebra $\equiv $ a bounded distributive lattice such that for all $a$ and $b$ there is a greatest element $x$ such that $a\wedge x\leq b$.

\item Implication operation, $a\Rightarrow b$ $\equiv $ the greatest element $x$ in a Heyting algebra such that $a\wedge x\leq b$.

\item Join-irreducible element $\equiv $ an element $x$ for which $x=y\vee z$ implies $x=y$ or $x=z$, for all $y,z$.

\item Meet-irreducible element $\equiv $ an element $x$ for which $x=y\wedge z$ implies $x=y$ or $x=z$, for all $y,z$.

\item Boolean algebra $\equiv$ a distributive lattice with a unary operation $\neg$ and nullary operations $0$ and $1$ such that $a\vee 0 = a$ and $a\wedge 1=a$, as well as $a\vee \neg a = 1$ and  $a\wedge \neg a = 0$.

\item Category $\equiv $ a class of objects and morphisms between them such that their composition is a well defined associative operation and that an identity morphism exists.

\item Functor $\equiv $ a map between two categories $A$ and $B$ that associates to each object of $A$ an object of $B$ and to each morphism in $A$ a morphism in $B$ so that the image of an identity morphism in $A$ is an identity morphism in $B$, and the image of the composition of morphisms in $A$ is the composition of their images in $B$.

\item Pullback $\equiv $ the limit of a diagram constituted by two morphisms with a common codomain.

\item Pushout $\equiv $ the colimit of a diagram constituted by two morphisms with a common domain.

\item Equalizer $\equiv $ the limit of the diagram consisting of two objects $X$ and $Y$ and two parallel morphisms $f, g : X \rightarrow Y$.

\item Coequalizer $\equiv $ the colimit of the diagram consisting of two objects $X$ and $Y$ and two parallel morphisms $f, g : X \rightarrow Y$ (dual concept of equalizer).

\end{itemize}


\section*{Acknowledgments}

The authors would like to thank to Karin Cvetko-Vah for several discussions on duality that helped clarifying some ideas presented here; to Mikael Vejdemo-Johansson for the suggestion of topos theory and the insights on its relevance to the foundations of persistence homology, to Margarita Ramalho for the relevant communications on topics of lattice theory and topology; and to Dejan Govc for the careful reading of this paper, his questions and his help on finding several typos in the first version submitted. 



\end{document}